\documentclass[onecolumn]{IEEEtran}

\ifCLASSOPTIONcompsoc
  \usepackage[nocompress]{cite}
\else
  \usepackage{cite}
\fi

\usepackage[utf8]{inputenc}
\usepackage{amsfonts,amsmath,amssymb,
amsthm, amsfonts, amsbsy}
\usepackage{mathtools}
\usepackage{bm}
\renewcommand{\overrightarrow}{\mathbf}
\newcommand{\matr}[1]{\mathbf{#1}}

\usepackage{array} 
\usepackage[usenames]{color}
\usepackage{graphicx}
\usepackage{paralist} 
\usepackage{flushend}

\usepackage[normalem]{ulem}
\usepackage{cancel}
\newcommand \mathstr[1]{{}}
\newcommand \rdout[1]{{}}

\DeclarePairedDelimiterX{\inp}[2]{\langle}{\rangle}{#1, #2}
\DeclareMathAlphabet{\mathcalligra}{T1}{calligra}{m}{n}

\newcommand \hide[1] {}
\newcommand\red[1]{{
#1}}

\newcommand\blue[1]{{
#1}}

\newcommand\green[1]{{
#1}}

\def \ham {\mathcal{H}}

\def \bLam {{\boldsymbol{\lambda}}}

\def \i {\item}
\def \bi {\begin{itemize}\item}
\def \ei {\end{itemize}}

\newtheorem{Assumption}{Assumption}
\newtheorem{Definition}{Definition}
\newtheorem{Proposition}{Proposition}
\newtheorem{Theorem}{Theorem}
\newtheorem{Lemma}{Lemma}
\newtheorem{Remark}{Remark}

\ifCLASSOPTIONcompsoc
 \usepackage[caption=false,font=footnotesize,labelfont=sf,textfont=sf]{subfig}
\else
  \usepackage[caption=false,font=footnotesize]{subfig}
\fi

\usepackage{stfloats}
\usepackage{url}

\begin{document}

\title{Spread, Then Target, and Advertise in Waves:\\ Optimal Budget Allocation Across Advertising Channels}

\author{Soheil~Eshghi,~\IEEEmembership{Member,~IEEE,} Victor~M.~Preciado,~\IEEEmembership{Member,~IEEE,} Saswati~Sarkar,\\ Santosh~S.~Venkatesh,~\IEEEmembership{Member,~IEEE,} Qing~Zhao,~\IEEEmembership{Fellow,~IEEE,} Raissa~D'Souza,\\ and Ananthram~Swami,~\IEEEmembership{Fellow,~IEEE}
\IEEEcompsocitemizethanks{\IEEEcompsocthanksitem S. Eshghi is with the School of Electrical Engineering, and Yale Institute of Network Science (YINS), Yale University, New Haven, CT 06520. This work was completed in his time at Cornell University.\protect\\
E-mail: {\tt\small soheil.eshghi@yale.edu}
\IEEEcompsocthanksitem V.M. Preciado, S. Sarkar, and S.S. Venkatesh are with the School of Electrical and Systems Engineering, University of Pennsylvania, Philadelphia, PA 19104.\protect\\
E-mail: {\tt\small preciado, swati, venkates@seas.upenn.edu}
\IEEEcompsocthanksitem Q. Zhao is with the School of Electrical and Computer Engineering, Cornell University, Ithaca, NY 15853.\protect\\
E-mail: {\tt\small qz16@cornell.edu}
\IEEEcompsocthanksitem R. D'Souza is with the Departments of Computer Science \& Mechanical and Aerospace Engineering, University of California, Davis, Davis, CA 95616.
E-mail: {\tt\small raissa@cse.ucdavis.edu}
\IEEEcompsocthanksitem A. Swami is with the Army Research Laboratory, Adelphi, MD 20783\protect\\
E-mail: {\tt\small ananthram.swami.civ@mail.mil}. %

}
\thanks{\red{This paper was presented in part in the 2017 IEEE Information Theory and Applications Workshop (ITA), San Diego, CA and in part as a poster at the 2017 Yale Day of Data, New Haven, CT.}
}}

\maketitle
\begin{abstract}
We analyze optimal strategies for the allocation of a finite budget that can be invested in different advertising channels over time with the objective of influencing social opinions in a network of individuals. In our analysis, we consider both exogenous influence mechanisms, such as advertising campaigns, as well as endogenous mechanisms of social influence, such as word-of-mouth and peer-pressure, which are modeled using diffusion dynamics.
 We show that for \rdout{a general set}\red{a broad family} of objective functions, the optimal influence strategy at every time uses all channels at either their maximum rate or not at all, i.e., a bang-bang strategy. Furthermore, we prove that the number of switches between these extremes is bounded above by a term that is typically much smaller than the number of agents. This means that the optimal influence strategy is to exert maximum effort in waves for every channel, and then cease effort and let the effects propagate. We also show that, at the beginning of the campaign, the total cost-adjusted reach of an exogenous advertising channel determines its relative value. In contrast, as we approach our investment horizon (e.g., election day), the optimal strategy is to invest in channels able to target individuals instead of broad-reaching channels.
We demonstrate that the optimal influence strategies are easily computable in several practical cases, and explicitly characterize the optimal controls for the case of linear objective functions in closed form. Finally, we see that, in the canonical example of designing an election campaign, identifying late-deciders is a critical component in the optimal design.
\end{abstract}

\section{Introduction}\label{sec:introduction}

\IEEEPARstart{O}{pinions} are important definers of real-world outcomes: they affect who is elected for political office \cite{gerber2009does}, which policies are successful \cite{brodie2011regional}, and which products are bought by customers \cite{chevalier2006effect}. The proliferation of online media has complicated \cite{campbell2014advertisements}, sped up \cite{pfeffer2014understanding}, and enhanced \cite{valenzuela2013unpacking} opinion formation processes. The opinion formation process can be affected by interested parties through \emph{advertising channels}, {which are media by which messages are distributed} to a target audience. Political campaigns and marketing departments apportion their advertising budgets between such channels (e.g., TV ads, website banner ads, billboards) in order to maximize some ultimate goal (e.g., votes, sales) \cite{chung2014air}{, though the extent of the effect of these efforts is a matter of debate \cite{huber2007identifying, gerber2009does}}. The importance of this decision has increased in conjunction with the increasing resources devoted to these efforts: {In 2017, over \$1 trillion was spent on marketing globally \cite{groupm2017}, while \$9.8 billion was spent on advertising in the 2016 US elections alone \cite{kaye2017data}.} Thus, studying the related multi-channel resource allocation problem is both timely and significant.

\begin{figure}[htb]
\begin{center}
\includegraphics[scale=0.6]{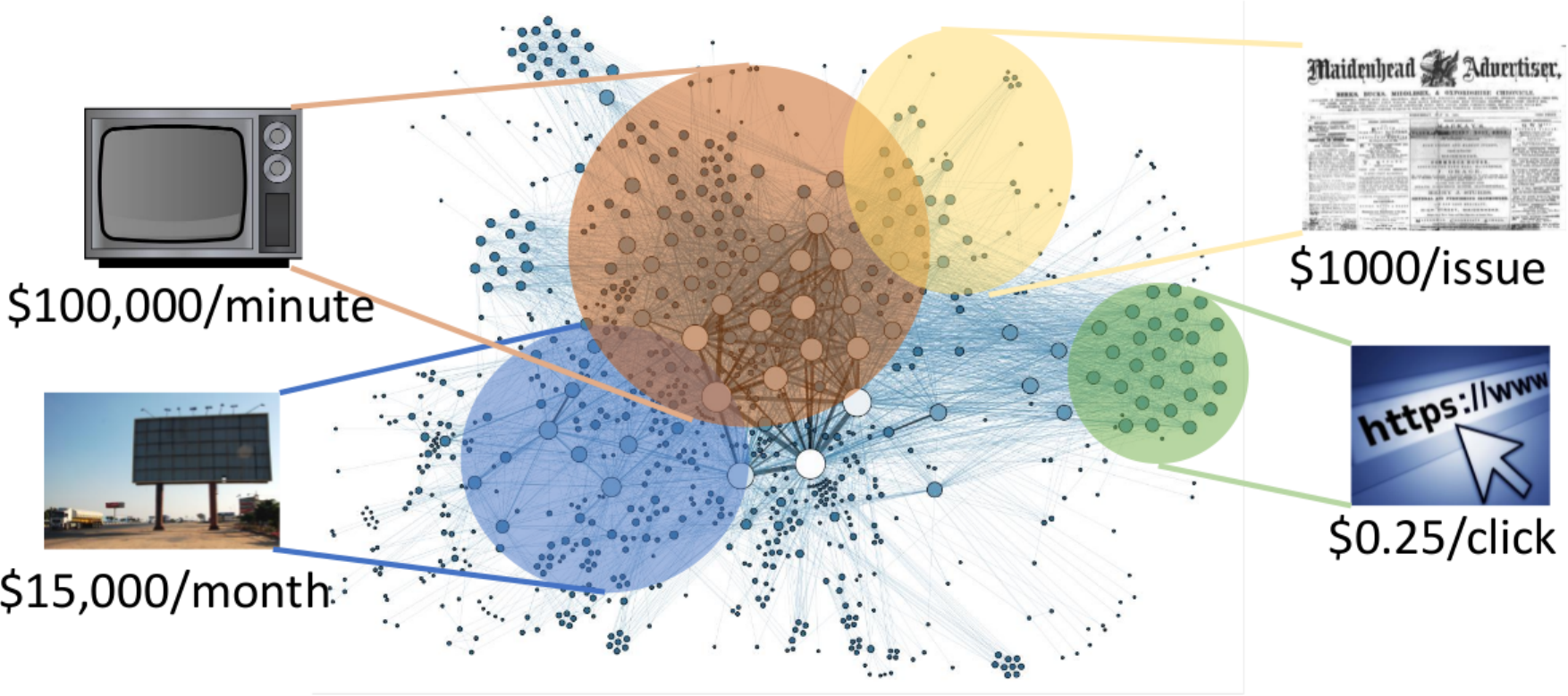}
\end{center}
\vspace{-0.1in}
\caption{\footnotesize{Different advertising channels have differing, possibly overlapping, reaches and come with differing costs. Furthermore, the effect of a channel may differ across individuals. These external influences are modulated by internal conversations within the network whereby agents integrate this information with that of their neighbors. The decision of the advertiser is to apportion resources between these channels factoring in these complexities so as to get an optimal return on investment, which in the election example is votes cast in favor of the campaign.}}\label{fig:reach}
\vspace{-0.1in}
\end{figure}

In particular, the mechanisms of opinion influence can be classified into two types based on its direct provenance. First, {there are} \emph{endogenous} influence mechanisms (e.g., word-of-mouth), in which individuals process the expressed opinions of other individuals they meet, and consider their credibility and the level of acquaintance and trust in synthesizing a new opinion based on the information.\footnote{These weights can, in general, be dynamic, even depending on the expressed opinion \cite{hegselmann2002opinion}. In this work, we consider static weights.} This leads to the notion of an endogenous influence weighted graph capturing endogenous influence between individuals. On the other hand, there are  \emph{exogenous} influence mechanisms, in which an external influencer seeks to shape the opinions of an individual. This mechanism is facilitated by various advertising channels\footnote{Throughout this work, we use the word channel to represent both the medium (e.g., TV advertising) and the reach of the medium (e.g., people who watch TV) -- the distinction is clear in context.}. In our opinion formation model, each channel has a reach structure, i.e. the individuals that can be reached by that channel, that is not necessarily related to the endogenous neighborhood. The actions of other external influencers also affect each individual's opinion formation process, which can in general be a random (noisy) process. The external influencer seeks to maximize a function of the \emph{global state} (the vector of individual opinions) at a specific time (e.g., election day) by allocating their budget  across several advertising channels in a given time interval (see Fig. \ref{fig:reach}). In this paper, we study the nature of the optimal budget allocation and provide structures and algorithms for their computation.

Finding the optimal budget allocation is complicated by several factors: (i) The reach of each channel is limited, and there are significant overlaps between the target audiences of various channels \cite{baines2011measuring}; (ii) different channels have differing costs, and attempts to influence opinions by external sources can affect individuals in different, and sometimes opposite, ways \cite{kahan2012polarizing}; (iii) the budget allocation decision is dynamic (depends on time) and changes with the state of the network\hide{, and therefore constitutes an optimization in the space of functions}. Furthermore, the influencer faces several trade-offs: utilizing an advertising channel early allows the influenced individuals to spread the effect to their neighbors (\emph{diffusion}), while lessening the impact on the influenced individuals as they moderate the effects of the external influence with the opinions of their neighbors (\emph{dilution}). There is also a trade-off between utilizing cheap channels versus utilizing expensive but effective ones. These competing forces make the \emph{a priori} determination of the optimal budget allocation hard to determine. 

There are also significant technical challenges to solving this problem, since characterizing the optimal budget allocation across channels and throughout the time interval requires characterizing the structure of an optimal constrained vector of controls over a graph. \hide{Currently, there does not exist a general theory to characterize the structure of vector optimal controls for general graph structure.} Furthermore, the work also  requires computing the optimal control of the well-studied linear consensus dynamics \cite{olfati2007consensus, spanos2005dynamic} in a novel setting, as the classical literature is concerned with reaching agreement among agents, while our objectives may incentivize agreement in some circumstances and disagreement in others. As we show in this paper, finding the optimal allocation in our problem requires a new synthesis of spectral graph theory and optimal control theory.  

\emph{Contributions:} In this work, we model the advertising influence problem as a constrained consensus control process in an arbitrary network with overlapping influence channels and endogenous influence of agents on each other \cite{naik2003understanding}. Using Pontryagin's maximum principle, spectral graph theory, and custom analytical arguments, we determine the structure of the optimal budget allocation to the various influence channels along a given time horizon. 

We show that for \rdout{a general class}\red{a broad family} of objectives, the optimal control for each channel is bang-bang (only takes its extreme values), with the number of switches being upper-bounded by a term\hide{\footnote{The number of non-zero elements in the orthogonal projection of the channel influence vector on the eigenvectors of the consensus dynamics.}}  which is \rdout{itself upper-bounded by}\red{smaller than} the number of individuals. Therefore, the search for optimal controls can be conducted on the space of vectors of a fixed size whose entries represent times of switching between extreme values rather than on the space of functions. Furthermore, for the case of a linear objective (i.e., when individuals make a decision in proportion to their opinion value), we explicitly calculate the optimal budget allocation over time, providing an open-loop algorithm that can compute the vector of optimal controls in a logarithmic number of steps. This allocation also implicitly determines the relative importance of a particular channel to the global objective, and thus defines {an explicitly computable metric for the influence of a channel} at any given time. This metric allows the influencer to compare and contrast the effects of different channels, as well as the effect of a channel at different times.
Finally, our results show that investing in an influence channel reaching likely voters is important as we get \emph{close to decision/election time}, while the \emph{cost-effectiveness} of a channel (defined as its total reach divided by its cost) is more important at earlier times.  
 
 For the case where the objective is a sum of sigmoids, which is a relaxed version of voting between two alternatives, we show that the optimal control can be approximated just by knowing the agents who change their minds at the terminal time in the optimal allocation (\emph{late-deciders} \cite{hayes1996marketing}). 
 
 In sum, our work represents a new confluence of the literature on consensus dynamics and optimal control theory, while providing significant novel structures, computational algorithms, metrics, and insights to the optimal budget allocation for the multi-channel advertising problem.

\section{Literature Review}

As this work draws upon the literature in multiple areas, we will discuss antecedents in each area in turn:

{\emph{Consensus and opinion dynamics}:
Linear consensus-seeking dynamics are some of the oldest models used to model the spread of opinions and social influence, first proposed by French \cite{french1956formal} and expounded upon by DeGroot \cite{degroot1974reaching}. In these models, opinions (states) are taken to be continuous real variables, and each node uses a (weighted) average opinion of its neighbors' opinions in each time-step to update its opinion. 
Abelson \cite{abelson1964mathematical} provided a continuous-time variant of these dynamics, which is the model we base our work upon. This work was generalized first by Taylor \cite{taylor1968towards} to also incorporate individual-specific prejudices, leading to the desirable persistence of opinion cleavage within such simple models (leading to a dynamics very similar to ours). Other closely related continuous-time variants of these dynamics have been rigorously studied by control theorists \cite{olfati2004consensus, olfati2007consensus, spanos2005dynamic, ren2005survey}. Most of these results focus on asymptotic properties of these dynamics and their convergence, and not on their finite-time behavior and the effect of influence on such behavior. While more complex models of opinion dynamics have been proposed and studied in detail \cite{hegselmann2002opinion, friedkin1990social, marsden1993network}, the linear consensus dynamics remains a baseline for comparison. Recent detailed overviews of the developments in the field of opinion dynamics make the above distinctions and limitations clearer \cite{friedkin2015problem, proskurnikov2017tutorial}.}
Finally, the linear approximation of the effect of external influence on opinion dynamics also follows a long-standing tradition \cite{taylor1968towards, friedkin1990social}.
{Our work covers finite-time budget-constrained opinion change with a specific goal, while the focus of these papers is understanding asymptotic properties of these systems (without strategic interventions and goals).}

{\emph{Control of Opinion Propagation}: The case of influencing opinion dynamics is a research question of current interest. The problem of Influence Maximization (IM) consists of finding the set of individuals that must be initially influenced in order to maximize the final effect of endogenous spreading mechanisms \cite{kempe2005influential}. Variants of this problem, under multiple models of opinion propagation, have been the subject of much study (e.g., \cite{mossel2007submodularity, anshelevich2009approximation,
borgs2014maximizing, morone2015influence}). Among this line of work, budgeted influence maximization with partial incentives \cite{demaine2014influence} is the closest to our setting, as it relaxes the artificial binary assumptions on the success of influence efforts.
While this literature is closely related to work on epidemic control \cite{morone2015influence}, its more immediate analog is work on control of social learning. For example, Yildiz {\em et al.} \cite{yildiz2011discrete} consider the case of stubborn agents who refuse to change their opinions in a two-opinion voter model. They show that the mean average opinion is only a function of the structure of the network and the placement of the stubborn nodes. They then investigate the optimal placement of these stubborn nodes.  However, the focus of all of these papers has been on static optimization, i.e., actions that are taken at a specific point in time. On the other hand, social networks are naturally dynamic, i.e., their states are time-varying, and it is natural to assume that actions prescribed to affect them can also be dynamic. In this paper, we analyze {such} optimal actions (henceforth referred to as controls) using tools from optimal control theory.}

\emph{Linear Optimal Control}:
\blue{In linear optimal control problems, a controller seeks to optimize the \rdout{accumulation}\red{time integration} of a linear objective \rdout{over time in a}\red{depending on the states and inputs of a} linear dynamical system with linear bounded controls.}

\blue{In the case where actions are not costly and the time horizon is not fixed,} 
the optimal control signal \rdout{presents}
\blue{has} a bang-bang structure with a finite number of switches \rdout{in the case where actions are not costly and the time horizon is not fixed} 
\cite{athans1966optimal, liberzon2012calculus, seierstad1986optimal}.
 However, these results do not apply directly in the case with costly actions and where the goal is not to drive the system to a known state in minimum time. In contrast, our work takes a step beyond those results and provides a context-specific method for evaluating the relative influence value on a channel within a time horizon. 

\emph{Optimal Control of Epidemic Spread and Diffusion}: This work bears a similarity with the literature on the optimal control of information spread, in that both aim to optimize a terminal function subject to some spread dynamics. Most such work uses compartmental epidemic models (e.g., SI \cite{kandhway2016campaigning, eshghi2015optimal}) and is thus dissimilar in dynamics to the one we consider. Furthermore, we show that when opinions can take continuous values (instead of the finite fixed values assumed in compartmental models), the optimal controls for influence maximization are significantly different to the strategies derived for information spread (which typically advocate some form of maximal spreading at the start of the time interval \cite{eshghi2015optimal, kandhway2014run}). The model also allows an even more explicit incorporation of graph structure than metapopulation models, e.g.,  \cite{eshghi2016optimal}, as their approximation breaks down when the population of each patch/type is small, and therefore provides a poor model for interactions at the scale of individuals.

\emph{Adversarial Sensor Network Deception}: {Finally, the problem discussed in this paper has a direct analog in the optimal deception of a sensor network by an adversary, as discussed in \cite{cardenas2008research}. In this setting, a state-estimation sensor network \cite{mo2010false} can be misdirected through local noise injection at a fixed number of points, that will affect a subset of nodes in the vicinity. The optimal locations and patterns for the noise to affect the conclusion of the network will depend also on the dynamic information fusion model of the sensor network and its relationship with the reach of each of the noise injection points. This problem, too, will require the same type of exogenous influence and endogenous processing model as the opinion influence problem, as well as having the same objective structure. Thus, any structural results obtained will have direct implications for the adversary's optimal deception policy.} {The modeling approach employed in our work is, to the best of our knowledge, novel for this setting.}

In summary, {our} work  integrates elements of the rich literature in linear consensus protocols, spectral graph theory, and optimal control, and applies the synthesis to the problem of resource allocation in advertising, achieving strong structural guarantees and applied insights.

\section{System Model Description}
In this section, we present our notation (\S \ref{subsec:notation}) and outline our system model (\S \ref{subsec:model}) and its dynamics (\S \ref{subsec:dynamics}). Then we outline the bounds on the actions of the influencer (\S \ref{subsec:bounds}) and describe their objective (\S \ref{subsec:objective}). We finish the section by presenting a technical assumption (\S
\ref{subsec:technical}) and by stating the overall problem (\S \ref{subsec:overall}).

\subsection{Notation}\label{subsec:notation}

\begin{align*}
n&=\text{number of agents}\allowbreak\\
m&=\text{number of channels}\allowbreak\\
x_i(t) &=\text{opinion of agent $i$ at time $t$},~~i =1,\ldots,n\allowbreak\\
u_k(t) &=\text{utilization of channel $k$ at time $t$},~~k =1,\ldots,m\allowbreak\\
u_k^{\max}(t) &=\text{maximum utilization of channel $k$ at time $t$}\allowbreak\\
a_{ij} &=\text{magnitude of effect agent $j$'s opinion on the opinion of agent $i$}\allowbreak\\
N_{i} &=\text{neighbors of agent $i$ in communication graph $G$}\allowbreak\\
H_{k} &=\text{set of agents within the reach of channel $k$}\allowbreak\\
b_{ik} &=\text{relative magnitude of the effect of utilization of channel $k$ on agent $i$}\allowbreak\\
e_i(t) &=\text{sum effect of other influences on agent $i$ at time $t$}\allowbreak\\
T &=\text{terminal time}\allowbreak\\
c_k(\cdot) &=\text{cost of influence on channel $k$}\allowbreak\\
r &=\text{total resources of influencer over the time period}\allowbreak\\
\mathcal{U} &= \text{set of feasible influence allocations}\\
J_{i}(\cdot) &=\text{value of opinion of agent $i$ at time $T$ to influencer}\allowbreak\\
J(\cdot) &=\text{value of opinion profile at time $T$ to influencer}
\end{align*}

\normalsize
We use bold lower case letters to denote vectors and bold upper case letters to denote matrices, $[n]$ to represent $\{1, 2,\ldots, n\}$, and $\inp*{a}{b}$ to represent $\overrightarrow{a}^T\overrightarrow{b}$.
For a matrix $\matr{W}$, we denote the $k$-th column of $\matr{W}$ as $\overrightarrow{W}(:,k)$, and the $k$-th row of the same as $\overrightarrow{W}(k,:)$. Furthermore, we use $w_{ij}$ to denote the $(i,j)$-th element of the matrix $\matr{W}$.

\subsection{System Model}\label{subsec:model}

We consider a social system with $n$ agents. The \emph{state/opinion} of agent $i \in [n]$ at time $t$ is denoted by $x_i(t)\in \mathbb{R}$. Each agent communicates with other agents based on an \textbf{edge-weighted}, \textbf{undirected}, and \textbf{connected} communication graph $G=(V,E,\matr{A})$\rdout{\footnote{\rdout{However, the results can be generalized to weighted directed graphs in which the weighted Laplacian has a generalized Jordan form with real eigenvalues}}}. 
The (non-negative) weight on an edge between agents $i, j \in [n]$, which determines the relative influence agent $j$ has on agent $i$'s state update, is represented by $a_{ij}$, and the matrix of such weights is represented by $\matr{A}$. An agent $j$ is said to be a \emph{neighbor} of agent $i$ (and vice versa) if $a_{ij}= a_{ji}>0$ (see Fig. \ref{fig:endo}).\footnote{While negative weight updates are conceivable, they will not be considered in this paper.} \red{The assumption of symmetric interaction weights, that a difference of opinion can have the same magnitude effect on both sides of an interaction, is common in the Influence Maximization literature, most commonly being present in the Independent Cascade (IC) model \cite{goldenberg2001talk, kempe2005influential, chen2009efficient}.} 

\red{\begin{Remark}\label{rem:asymmetric}
However, weighted \textbf{directed} communication graphs can also be considered in our framework, in which case \green{some of} our results apply to cases where the weighted Laplacian of the graph has\rdout{ a generalized Jordan form with} real eigenvalues.\footnote{\red{Lemma \ref{lem:existence} and Theorem \ref{thm:crossing}.\ref{thm:first} carry over, as does a modification of the water-filling procedure in \S \ref{subsec:waterfilling}.} \green{For specifics, see \S \ref{sec:directed}}.} \green{In particular, this includes the set of quasi-strongly connected weighted Directed Acyclic Graphs (DAGs).}
\end{Remark}}

At each time $t$, each agent updates its state based on a \emph{weighted average} of the difference of its current state from those of its neighbors, as well as on an external influence that will be described below, and a known drift signal (which may be due to the influence of other competing influencers), which we denote by $e_i(t)$ for $i\in [n]$. 

\begin{figure}
\begin{center}
\includegraphics[scale=0.35]{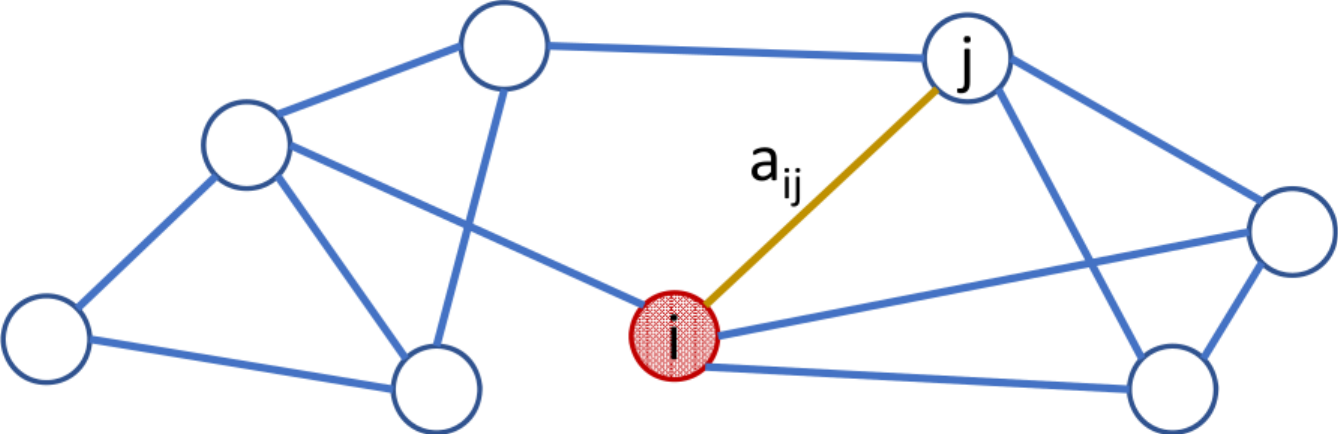}
\end{center}
\vspace{-0.1in}
\caption{\footnotesize{Each agent takes into account the opinions of its neighbors in updating its own opinion. The weight given by node $i$ to the opinion of its neighbor $j$ is a measure of how much $i$ trusts $j$'s appraisal. If agents $i$ and $j$ are not neighbors, $a_{ij}=0$ by default. In our model, $a_{ij}= a_{ji}$, i.e., trust is a symmetric relation\red{, though our results apply even if that is not the case (see conditions in Remark \ref{rem:asymmetric}).}}}\label{fig:endo}
\vspace{-0.1in}
\end{figure}

An influencer aims to shape the \emph{opinion profile} (i.e., the opinion vector of all agents) at a fixed terminal time $T$ according to an objective function through the judicious use of particular \emph{influence channels}. Each channel of influence (e.g., advertising medium) is limited in its reach, as it only affects a specific subset of agents (denoted $H_k$ for channel $k$). \hide{The structure of these channels is pre-specified and captured by a hyper-graph $G'=(V, H)$ on the same node set.}\hide{ There are $|H|= m$ hyper-edges (channels) that can be influenced, with the assumption that influencing a channel only directly affects the members within that channel.}
{The structure of these $m$ influence channels is pre-specified, with the assumption that influencing a channel only directly affects the members within that channel.} The influence exerted by the influencer on channel $i\in [m]$ at time $t$ is denoted by the scalar $u_i(t)$. 

In this model, the effect of influence on a channel can differ across agents within the channel, potentially even having opposite effects. These effects are captured by the \emph{influence gain}, denoted by $b_{ik}$, which determines the linear relative gain of influence of channel $k \in [m]$ on agent $i$ within that channel. \blue{For example, if billboard advertising (say, channel $k$) has a more positive effect on the opinion of individual $i$ than advertising on the radio (say, channel $l$), we will have $b_{ik}> b_{il}$.}\footnote{\red{Note that the magnitude of $b_{ik}$ can be determined by comparing the size of the effect of channel $k$ on individual $i$'s opinion with that of one of $i$'s neighbors having the same amount of difference in opinion with $i$.} \blue{At scale, these orderings and values can possibly be inferred from demographic information.}} If agent $i$ is not within channel $k$, we define $b_{ik}$ to be zero. Without loss of generality, we assume $u_i\geq 0$ for $i \in [m]$, {and encode the possible negative effects of channel $i$ on agent $k$ within its reach through the sign of $b_{ik}$}. Stacking these values into a matrix $\matr{B}_{n*m}$ {captures the \hide{hypergraph} structure of the channels.} 

\subsection{Dynamics}\label{subsec:dynamics}

To understand the dynamics, we provide the following discrete-time intuition: an agent $i\in [n]$ constructs its change in state in the time interval $(t,t+\Delta)$ based on the weighted difference between its own state and that of its neighbors, as well as the external influence exerted on it in that time period and the drift signal:
\begin{align*}
x_i(t+&\Delta)=x_i(t)+\Delta \bigg(\sum_{j\in N_i}a_{ij}(x_{j}(t)-x_{i}(t)) + \sum_{k:i\in H_k}b_{ik}
u_{k}(t)
+ e_{i}(t)\bigg).
\end{align*}
This simply states that agents attempt to align their state/opinion with that of their neighbors, and the influencer's effort may act as a hindrance to that process. 
Mathematically, it can be thought of as a gradient descent algorithm implemented by agents seeking to minimize disagreement (measured by a Laplacian potential) \cite{olfati2004consensus}. The above can be re-written to represent the classic discrete-time consensus model \cite{degroot1974reaching} with influence:\footnote{The results derived in this paper would also apply to the Friedkin and Johnsen model of opinion updates \cite{friedkin1990social} given uniform susceptibility to change across agents.}
\begin{align*}
x_i(t+\Delta)=&(1- \Delta \sum_{j\in N_i} a_{ij}) x_i(t)+ \Delta \sum_{j\in N_i}  a_{ij} x_{j}(t) + \Delta  \bigg(\sum_{k:i\in H_k}b_{ik}
u_{k}(t)
+ e_{i}(t)\bigg).
\end{align*}
\red{Note also that, from this formulation, it is evident that the influencer's effect on the state of any individual is dissipative, as their prior state is discounted by a factor of $(1- \Delta \sum_{j\in N_i} a_{ij})<1$ at every time-step.}\footnote{
\blue{This can be seen by looking at the explicit effect after two time-steps (i.e., time $t+ 2\Delta$), where the effect of direct influence exerted at time $t$, $\overrightarrow{u}(t)$, is multiplied by $(1- \Delta \sum_{j\in N_i} a_{ij}) $:
$x_i(t+2\Delta)=\ldots + \Delta \sum_{k:i\in H_k} b_{ik}u_k(t+ \Delta) + \underline{\mathbf{(1- \Delta \sum_{j\in N_i}} a_{ij})} \Delta  \sum_{k:i\in H_k} b_{ik}u_k(t)+ \Delta^2 \big(\sum_{j\in N_i}\sum_{m:j\in H_m} a_{ij}b_{jm}u_{m}(t)\big) + \ldots$
}}

Subtracting $x_i(t)$ from both sides, dividing by $\Delta$ and taking the limit as $\Delta$ goes to zero, we arrive at the following continuous time agent-level dynamics:\footnote{\blue{While discrete-time dynamics are more commonly used for the modeling of opinions, discretization is typically a simplifying assumption for analytic purposes. In this paper, we work with the continuous-time dynamics directly, which allows the use of mathematical tools new to the domain. However, all derived structures and insights for the continuous case can be discretized and applied to the discrete-time case as well.}}
\begin{align*}
&\dot{x}_{i}(t)=\sum_{j\in N_i}a_{ij}(x_{j}(t)-x_{i}(t)) + \sum_{k:i\in H_k}b_{ik}u_{k}(t) + e_{i}(t).\hide{\\&=(-\sum_{j\in N_i} a_{ij})x_{i}(t) + \sum_{j\in N_i}a_{ij}x_{j}(t) + \sum_{k:i\in H_k}b_{ik}
u_{k}(t)
+ e_{i}(t)}
\end{align*}
We let $\matr{L}$ be the weighted Laplacian matrix, where for all $i,j \in [n]$ such that $i\neq j$, $l_{ij}= - a_{ij}$, and $l_{ii}= \sum_{j\in N_i} a_{ij}$ for all $i \in [n]$.
Stacking the $n$ equations, we arrive at the following system-level dynamics:
\begin{align}\label{eq:vecdynamics}
 \dot{\overrightarrow{x}}(t)= - \matr{L}\overrightarrow{x}(t) + \matr{B}
\overrightarrow{u}(t)
 + \overrightarrow{e}(t)
\end{align}
We assume that the  states/opinions at time $0$ are known ($\overrightarrow{x}(0)=\overrightarrow{x_0}$), however we will see that the value of the states at time $0$ has no direct bearing on our structural results.

\subsection{{Admissible Control Strategies}}\label{subsec:bounds}

The total expenditure on all channels is bounded by $r>0$, {which is the budget available to the influencer}. This is captured through the following budget constraint:
\begin{align}\label{eq:influencecapital}
\int_0^T \sum_{k=1}^m c_k(u_k (t))\,dt\leq r ~~&~~\text{(budget {constraint})},
\end{align}
where $c_k(\cdot)$ represents the \red{time-independent} cost-function \rdout{associated with  investing in channel $k$}\red{which maps the utilization of channel $k$ to its associated cost to the influencer}.
\begin{Assumption}\label{ass:increasing}
We assume that for all $k \in [m]$, $c_k(\cdot)$ is increasing, differentiable and concave \red{as a function of channel $k$'s utilization}. Furthermore, without loss of generality, we assume that $c_k(0)=0$ for all $k\in [m]$.
\end{Assumption}
This assumption models the diminishing cost of additional utilization of a channel once it is already in use. 
The above assumption allows the case where $c_k$ is linear (that is $c_k(x)= v_k x$). We assume that for all channels  $k\in [m]$, the influence that can be exerted on channel $k$ at any time $t$ is bounded above by a time-varying value $u_k^{\max}(t)
$.\footnote{This rules out impulse controls.} This can capture both physical limits on the influence {(i.e., availability of media)} and limits on the susceptibility of agents to the influence. We impose the modest assumption that $u_k^{\max}(t)$ is differentiable. Hence, we have that:
\begin{align}\label{eq:influencebounds}
 \forall {k\in [m]},~~~~0\leq u_k (t) \leq u_k^{\max}(t)~
\text{(influence constraint)}
\end{align}
{We will restrict our analysis to control signals $\overrightarrow{u}$ that are piecewise continuous with only a finite number of discontinuities\footnote{This means, in particular, that the integral in \eqref{eq:influencecapital} is well-defined.}. We shall use $\mathcal{U}$ to denote the set of such controls that fulfill \eqref{eq:influencebounds}:}
\begin{align*}
\mathcal{U}=\{\overrightarrow{u}: 0\leq u_k (t) \leq u_k^{\max}(t)
, k \in [m], t\in [0,T]\}.
\end{align*}
{In our model, employing channel $k \in [m]$ at effort level $u$ at time $t \in [0,T]$ incurs a cost of $c_k(u)$, where for all $k\in [m]$, $c_k: {[0, + \infty)\hide{\max_{t\in [0,T]}u_k^{\max}(t)]}} \to [0, +\infty)$. Note that we assume that for all $k\in [m]$, the cost function $c_k(\cdot)$ is time-invariant.}

 \hide{influencing each channel at a particular time has a cost function that depends on the channel and takes as input the level of influence that the influencer exerts on that channel. We assume that this function is fixed over the time interval.}\hide{\footnote{The case where cost functions are time-dependent (i.e., TV advertising) can also be analyzed as presented, however the obtained results will be weaker.} }

\subsection{Objective}\label{subsec:objective}

The objective of the influencer is a function of the opinion profile at a fixed time $T$. The nature of the function will depend on the information aggregation method employed by the set of individuals. We consider the most general case, where any increase in the opinion of any particular individual at time $T$ (keeping all other opinions the same) is not detrimental to the influencer. This is obviously the case in both political and marketing campaigns. While our reasoning applies to a general family of objective functions, we will give special consideration to functions that model voting in an election between two options (relevant in the political campaign setting) and weighted averaging (relevant in estimating total returns from marketing efforts and in the sensor network setting).

\begin{Assumption}
The objective, $J(\overrightarrow{x}(T))$, is an increasing\hide{\footnote{This encodes the non-negative benefit of an increase in opinions at time $T$ to the influencer.}}, differentiable function of the $n$ components of the vector of terminal opinions $\overrightarrow{x}(T)$. 
\end{Assumption}
\hide{This is commonly known in the optimal control literature as a Mayer-type cost \cite[page 70]{liberzon2012calculus}.} In particular, we will elaborate on the application of our results to a particular family of objective functions that are separable in the elements of the vector of opinions, as follows:
\begin{align}\label{obj:separable}
J(\overrightarrow{x}(T)) = \sum_{i=1}^n J_i(x_i(T)).
\end{align}
Two specific types of separable functions are of particular interest:

1) {\bf Linear functions}:
\begin{align}\label{obj:linear}
J_i(x_i(T))=p_i x_i(T), ~~~ p_i>0,
\end{align}
which model the simplest case, where the utility the influencer gains from an individual has a linear relationship with its state at time $T$. In the marketing example, this can model the amount of sales as a simple function of an individual's opinion of a product. This is also a useful approximation in the adversarial sensor network deception case where the utility for the influencer is a simple function of a sensor's report level. The mapping of the election example to this utility is not direct: this can model the case where each agent  votes with probability $p_i$, and if it does, chooses among two options with a probability that is linearly related to their opinion (i.e., they flip an appropriately weighted coin). However, finding the correct total weight of the coin to be considered depends on the assumptions made by the modeler, and multiple normalizations may be defensible. This ambiguity leads to the definition of a second type of utility for the specific case of the election example.

2) {\bf Sigmoid functions}:
Assume each individual $i \in [n]$ has to vote for one of two options (e.g., candidates, products, policies), encoded by $0$ and $1$, at time $T$. Assume that the  
influencer backs option $1$ (without loss of generality). 
Each individual is assumed to vote with probability $p_i>0$, and to choose who to vote for among two options based on whether their state at time $T$ is above or below an agent-specific threshold $\theta_i$ (which models the various biases for and against an option). Thus, the utility gained from each individual can be modeled using a Heaviside function with a jump at $\theta_i$, which is agent $i$'s vote. However, this utility is discontinuous at $x_i(T)=\theta_i$, which complicates analysis.
The sigmoid function\hide{(see Fig. \ref{fig:sigmoid})}:
\begin{align}\label{obj:sigmoid}
J_i(x_i(T))=\frac{p_i}{1+e^{-\alpha_i(x_i(T)-\theta_i)}},
\end{align} 
 is a smooth approximation to the Heaviside utility, with the closeness of the approximation being determined by the choice of the parameter $\alpha_i$ --- the greater $\alpha_i$ is, the faster the transition. In the extreme of taking $\alpha_i$ to infinity, this function will indeed converge to the aforementioned Heaviside function.

\subsection{Technical Assumption}\label{subsec:technical}

We now add a technical assumption that will be needed in our arguments:
\begin{Assumption}\label{ass:positive}
{There exists a $j\in [n]$ such that $\dfrac{\partial J(\overrightarrow{z})}{\partial z_j(T)}>0$ for all $\overrightarrow{z}\in \mathbb{R}^n$.}
\end{Assumption}
Note that this is equivalent to saying there exists at least one individual such that the influencer always values a marginal increase in its state. That is, holding all opinions the same, any increase in that agent's opinion will be translated to a strict increase in their likelihood of voting for the choice backed by the influencer.\footnote{{This rules out $J(\cdot)$ functions with stationary points, i.e., those for which $\nabla_{\overrightarrow{z}} J(\overrightarrow{z})= \overrightarrow{0}$ for some $\overrightarrow{z}\in \mathbb{R}^n$. For example, this rules out an objective which is a sum of Heaviside functions.}} The purpose of this assumption is to rule out a pathological case where the necessity conditions for the optimality of an allocation become so general that they apply to all controls and are thus uninformative.

\subsection{Problem Statement}\label{subsec:overall}

We aim to characterize the control inputs $\overrightarrow{u}(t)$ that maximize $J(\overrightarrow{x}(T))$ under the dynamics outlined in \eqref{eq:vecdynamics} and constraints \eqref{eq:influencebounds} and \eqref{eq:influencecapital}. Mathematically, we state our problem as:

\begin{align*}
\max_{\overrightarrow{u}\in\mathcal{U}} ~~~~&J(\overrightarrow{x}(T)) \notag\\
\text{s.t.} ~~~~& \dot{\overrightarrow{x}}(t)= - \matr{L}\overrightarrow{x}(t) + \matr{B}\overrightarrow{u}(t)+ \overrightarrow{e}(t),~~ \overrightarrow{x}(0)= \overrightarrow{x_0}\in \mathbb{R}^n \\
\hide{&0\leq u_m (t) \leq u_m^{\max},
~~~~ \forall {m},\\}
&\int_0^T \sum_{k=1}^m c_k(u_k (t))\,dt\leq r,\\
& \overrightarrow{e}(t) \text{~~given}.
\end{align*}
{Note that the above problem is non-convex in general owing to the potentially non-convex objective function, as well as the potentially non-convex budget constraint (when any of the $c_i(\cdot)$'s are strictly concave). In this paper, we solve it using tools from optimal control theory.} 
\hide{This problem has a general-form, potentially non-convex objective function, as well as a potentially non-convex budget constraint (when any of the $c_i(\cdot)$'s are strictly concave).} It should be observed that the number of competing influence channels, $m$, and therefore the number of optimization variables, can potentially be large. These factors complicate naive approaches to solving the problem.

We reformulate the problem with auxiliary variables to aid the analysis. We define the auxiliary functions $\gamma$ and $q$ such that,
\begin{align}
\gamma(0)&=0, &\dot{\gamma}(t)= - \sum_{k=1}^m  c_k(u_k(t)), \label{eq:newstates}\\
{q(0)}&={0}, &{\dot{q}(t)=1 \label{eq:newstates1}.\hide{, ~(\text{proxy for time})}}
\end{align}
As can be seen, $\gamma(t)$ is the accumulated cost of the influence up to time $t${, and} ${q}$ {is a proxy for time}\footnote{\green{The problem, as originally stated, is \emph{non-autonomous}, i.e., the dynamics depend explicitly on the independent variable through the function $\overrightarrow{e}(t)$. It can be simplified to the \emph{autonomous} case, which is more suitable for computation, by removing explicit time-dependence in the dynamics through the introduction of the dummy variable $q$ that is always equal to time \cite[page 167]{seierstad1986optimal}.}}.
Thus, the budget constraint becomes $\gamma(T)\geq-r$, and the integral constraint has been transformed to a terminal time one. So we can rewrite the optimization as:
\begin{align}
\max_{\overrightarrow{u}\in \mathcal{U}} ~~~~&J(\overrightarrow{x}(T)) \label{prob:reformulated}\\
\text{s.t.} ~~~~& \dot{\overrightarrow{x}}(t)= - \matr{L}\overrightarrow{x}(t) + \matr{B}\overrightarrow{u}(t)
 + \overrightarrow{e}({q}),~~ \overrightarrow{x}(0)= \overrightarrow{x_0}\in \mathbb{R}^n,\notag\\
&\dot{\gamma}(t)= - \sum_{k=1}^m  c_k(u_k(t)){, ~~~\dot{q}(t)=1}, ~~~\gamma(T)\geq-r,\notag\\
& \overrightarrow{e}(t) \text{~~given}, ~~\gamma(0)=0{,~~ q(0)=0}.\notag
\end{align}

\section{Results}
In this section, we outline the analytical structures of the optimal controls. To show the nature of the results, we first explain some necessary priors in \S \ref{subsec:pre}. Then, we prove the existence of optimal controls (under some conditions) and identify their structure using our main theorem in \S \ref{subsec:main} (with proofs in \S Appendices \ref{sec:lemmaexist} and \ref{subsec:proof1}, respectively). A refinement is presented for the case of the linear objective \S \ref{subsec:waterfilling} that allows the direct computation of the control input $\overrightarrow{u}$ and shows that the optimal control is unique, while providing insights into the logic of the allocation decision. Finally, the sigmoid approximation to voting is covered in \S \ref{subsec:sigmoid} and an approximation to the optimal control is presented.

\subsection{Preliminaries}\label{subsec:pre}

For an {\bf undirected, connected} graph $G$, the weighted Laplacian matrix $\matr{L}$ is real, symmetric, and positive semi-definite; hence it has real, non-negative eigenvalues \cite[page 13]{chung1997spectral}. Thus, $\matr{L}$ has an eigen-decomposition $\matr{L}=\matr{Q} \matr{\Xi} \matr{Q}^T$, where $\matr{Q}$ is a \red{real} orthogonal matrix whose columns are the eigenvectors of $\matr{L}$, and $\matr{\Xi}$ is the diagonal matrix of eigenvalues of $\matr{L}$) \red{\cite[p. 393, Theorem 8.1.1]{golub2012matrix}}.\footnote{However, the reasoning below applies to {any $\matr{L}$ that has \rdout{a generalized Jordan form with }real eigenvalues} \green{for the case of linear costs $c_i(\cdot)$ and a quasi-strongly connected communication digraph. For details, see \S \ref{sec:directed}.}} The smallest eigenvalue of $\matr{L}$ is always \textbf{zero}, its multiplicity is 1, and its associated eigenvector is $\frac{1}{\sqrt{n}}\overrightarrow{1}_{n}$, where $\overrightarrow{1}_{n}=(1,\ldots,1)^T$ (as $G$ is connected) \cite[page 13]{chung1997spectral}. We will order the eigenvalues of $\matr{\Xi}$ smallest to largest ($\xi_1=0<\xi_2\leq\ldots\leq \xi_n$) and, therefore, column $i\in [m]$ of $\matr{Q}$, $\overrightarrow{Q}(:,i)$ will be the $i$-th eigenvector of $\matr{L}$\hide{\footnote{Notice this is possible because $\matr{L}$ is Hermitian and therefore the geometric multiplicity of all eigenvalues is equal to their algebraic multiplicity \cite[page 6-5]{hogben2006handbook}.}}. This means that $\overrightarrow{Q}(:,1)=\frac{1}{\sqrt{n}}\overrightarrow{1}_{n}$.

We now state a lemma that shows that an optimal control $\overrightarrow{u}$ for the main problem exists. We then state our main result (Theorem \ref{thm:crossing}) and present a subcase where the bound can be significantly strengthened and the optimal control can be calculated in open-loop (Theorem \ref{thm:linear}). We provide proofs of these results in \S Appendix \ref{sec:lemmaexist} and \S Appendix \ref{subsec:proof1}, respectively. 

\subsection{Existence of Optimal Solutions and Structural Results for the Optimal Control}\label{subsec:main}

We prove (in \S Appendix \ref{sec:lemmaexist}) that optimal controls for \eqref{prob:reformulated}, exist when $c_i(\cdot)$'s are linear.
\begin{Lemma}\label{lem:existence}
Optimal controls for problem \eqref{prob:reformulated} exist for linear costs, i.e., $c_i(u_i)=v_i u_i$ for all $i\in [m]$.
\end{Lemma}
We are now ready to state our main theorem. We will provide results for a large natural class of channels that we shall call \emph{disciplined}. We first formally define the set of {disciplined} channels $\mathcal{I}$ before stating the theorem:
\begin{Definition}\label{def:disciplined}
The set of disciplined channels, $\mathcal{I}\subseteq [m]$, is such that for all $i \in \mathcal{I}$, one of the two following conditions holds: 
\begin{compactitem}
\item $c_i(\cdot)$ is strictly concave, $\inp*{\overrightarrow{B}(:,i)}{\overrightarrow{{1}}_n}\neq 0$, and $u_i^{\max}(t)=u_i^{\max}$ for all $t$.
\item $c_i(\cdot)$ is linear and the system $(\matr{L}, \matr{L}\overrightarrow{B}(:,i))$ is controllable \cite[p. 144]{chen1995linear} (with $u_i^{\max}(t)$ being any differentiable function).
\end{compactitem}

\end{Definition}

\begin{Theorem}\label{thm:crossing}
For all $i\in \mathcal{I}$: 
\begin{enumerate} 
\i\label{thm:first} Optimal controls are \textbf{\emph{bang-bang}}, taking on their maximum or minimum values at all times $t$ (i.e., $u^*_i(t)\in \{u^{max}_i, 0\}$). 

\i The number of switches between these values is bounded above:

\begin{enumerate}
\i\label{thm:parta} In the general case, by one less than the number of non-zero elements in $\{\inp*{\overrightarrow{Q}(:,j)}{\overrightarrow{B}(:,i)}\}_{j=1}^n$.

\i\label{thm:partb} For $J(\overrightarrow{x}(T))=\inp*{\overrightarrow{p}}{\overrightarrow{x}(T)}$, by the number of sign variations in $\{\sum_{k=1}^j s_k \}_{j=1}^n$, where $s_j : = \inp*{\overrightarrow{Q}(:,j)}{\overrightarrow{p}} \inp*{\overrightarrow{Q}(:,j)}{\overrightarrow{B}(:,i)}$. 
\end{enumerate}
\end{enumerate} 

\end{Theorem}

An example of an optimal control with these characteristics is provided in Fig. \ref{fig:bangbang}. The proof of this theorem is presented in \S Appendix \ref{subsec:proof1}. 

This theorem means that the optimal strategy uses each channel in waves (see Fig. \ref{fig:bangbang}), stopping between them to let influence propagate. From a computational stand-point, this result simplifies the space of possible optimal controls for each channel, since the optimal control is characterized by the bounded number of switching times for each channel. The actual number of switches of each optimal control can be significantly less than the fixed upper-bound of $n-1$ (which can in general be very high), as we will see in \S \ref{sec:simulations}.

\begin{Remark}
The conditions in Definition \ref{def:disciplined} rule out pathological cases where the necessary conditions for optimality derived from the Maximum Principle \cite[page 182]{seierstad1986optimal} cannot directly determine the optimal value of the control (i.e., singular arcs \cite[page 113]{liberzon2012calculus} exist).
\end{Remark}

\begin{Remark}\label{rem:disciplined}
The set of disciplined channels may be a proper subset of the set of channels ($\mathcal{I}\subset [m]$), in which case the derived structure only applies to disciplined channels. This means that even if $c_k(\cdot)$ is non-concave or  the conditions around $\overrightarrow{B}_k$ in Definition \ref{def:disciplined} do not hold for some $k$, \textbf{Theorem \ref{thm:crossing} will remain valid for disciplined channels}. Note that optimal controls for undisciplined channels may also abide by the bang-bang structures stated in Theorem \ref{thm:crossing}.
\end{Remark}

\begin{figure}[htb]
\begin{center}
\includegraphics[scale=0.82]{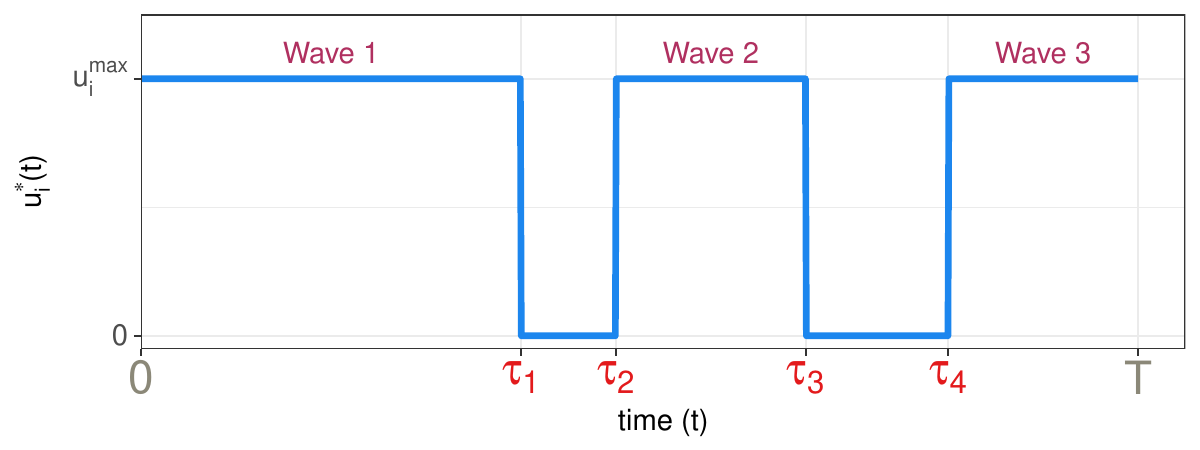}
\vspace{-0.1in}
\caption{\footnotesize{For a $c_i(\cdot)$ and $\protect\overrightarrow{B}(:,i)$ fulfilling the conditions of Theorem \ref{thm:crossing}, the optimal control $u_i(t)$ will be bang-bang, only taking its minimum or maximum values and switching between them a bounded number of times. Thus, the function can be fully described by the set of switching times $\{\tau_i\}_i$, {making them easier to compute, store, and implement.}}
}
\label{fig:bangbang}
\end{center}
\vspace{-0.2in}
\end{figure}

\subsection{Water-filling: Optimal Budget Allocation for Separable Linear Objectives}\label{subsec:waterfilling}

In this section, for separable linear objectives, we will derive a detailed cost-effectiveness metric for channel $i$'s utilization that depends on the eigenvalues and eigenvectors of the Laplacian ($\matr{L}$), the channel influence gain vector ($\overrightarrow{B}(:,i)$), and the weights of the linear objective ($\overrightarrow{p}$). The variation of this metric across time will result in \emph{hills} and \emph{valleys} that represent the variations in the effectiveness of the channel across time. Choosing a water-line for this topography (see Fig. \ref{fig:waterfilling}), we will show, leads to the description of a candidate control which takes its maximum values when a hill is above water, and will be set to zero when a valley is under water. This waterline is varied {using the bisection/binary search method} so that the cost of the total area above water matches the budget constraint \eqref{eq:influencecapital}. We will further show how this approach generalizes for more varied objective functions.

From the proof of Theorem \ref{thm:crossing} in \S Appendix \ref{subsec:proof1}\hide{\footnote{ \eqref{eq:explicitdependence} coupled with the Hamiltonian-maximizing condition of the Maximum Principle leads to \eqref{eq:cases-concave} and \eqref{eq:cases-linear} as explicit necessary conditions for the optimal controls. We then replace \eqref{eq:equation} explicitly into the aforementioned necessary conditions for the case of $J(\overrightarrow{x}(T))=\inp*{\overrightarrow{p}}{\overrightarrow{x}(T)}$.}}, we can define:
\begin{align}\label{eq:weights} h_i(t)&=\inp*{\overrightarrow{\bLam^*}(t)}{\overrightarrow{B}(:,i)}=\sum_{j=1}^N\inp*{\overrightarrow{Q}(:,j)}{\overrightarrow{p}}\inp*{\overrightarrow{Q}(:,j)}{\overrightarrow{B}(:,i)}e^{-\xi_{j}(T-t)},
\end{align}
where $\overrightarrow{Q}(:,j)$ is the $j-$th eigenvector of the Laplacian matrix $\matr{L}$ with associated eigenvalue $\xi_j$, and $\overrightarrow{p}$ is the vector of weights of the linear objective, i.e., $J(\overrightarrow{x}(T))= \inp*{\overrightarrow{p}}{\overrightarrow{x}(T)}$, such that the necessary condition for optimal controls in the concave $c_i(\cdot)$ case becomes \footnote{\red{The question mark denoting the fact that PMP does not uniquely determine the optimal $u_i^*$ at times $t$ when $\varphi_i(t, u_i)$ does not change with $u_i$.}}:
\begin{align}\label{eq:cases-concave1}
\hspace*{-0.3in}u^*_i(t)=
\begin{cases}
u^{\max}_i, ~ &\text{if} ~h_i(t)> \beta^*(T) \dfrac{c_i(u^{\max}_i)}{u^{\max}_i},\\
0, ~  &\text{if}~ h_i(t)< \beta^*(T) \dfrac{c_i(u^{\max}_i)}{u^{\max}_i},\\
\text{?}, ~ &\text{if}~ h_i(t)= \beta^*(T) \dfrac{c_i(u^{\max}_i)}{u^{\max}_i},
\end{cases}
\end{align}
and for the linear $c_i(\cdot)$:
\begin{align}\label{eq:cases-linear2}
u^*_i(t)=
\begin{cases}
u^{max}_i(t), \quad &\text{if} \quad h_i(t) > \beta^*(T) v_i,\\
0, \quad &\text{if} \quad  h_i(t)  < \beta^*(T) v_i,\\
\text{?},\quad &\text{if} \quad h_i(t)  = \beta^*(T) v_i,
\end{cases}
\end{align}
for some optimal \emph{a priori} unknown parameter $\beta^*(T)$. 
All other terms in \eqref{eq:cases-concave1} and \eqref{eq:cases-linear2} are explicitly computable without solving the optimal control problem. Thus, determining $\beta^*(T)$ will determine $\overrightarrow{u}(t)$ for all $t$ except for a finite, explicitly bounded number of points (notice that the existence of singular controls was ruled out in the proof of Theorem \ref{thm:crossing}). However, as we shall see in \eqref{eq:terminal2} \red{of the appendix}, $\beta^*(T)>0$ if and only if
$\int_0^T \sum_{k=1}^m c_k(u_k (t))\,dt= r$ \hide{However, $\beta^*(T)$ has to satisfy \eqref{eq:terminal2}, which means $\beta^*(T)>0$ if and only if $\gamma(T)+r =0$, or
$\int_0^T \sum_{k=1}^m c_k(u_k (t))\,dt= r$.} This last equation is the budget constraint.

Define the equivalent of \eqref{eq:cases-concave1} and \eqref{eq:cases-linear2} as functions of a variable $\hat{\beta}(T)$, an estimate for $\beta^*(T)$:
\begin{align}\label{eq:cases-concave3}
\hspace*{-0.3in}u_i(t, \hat{\beta}(T))=
\begin{cases}
u^{\max}_i, ~ &\text{if} ~h_i(t)> \hat{\beta}(T) \dfrac{c_i(u^{\max}_i)}{u^{\max}_i},\\
0, ~  &\text{if}~ h_i(t)< \hat{\beta}(T) \dfrac{c_i(u^{\max}_i)}{u^{\max}_i},\\
\text{?}, ~ &\text{if}~ h_i(t)= \hat{\beta}(T) \dfrac{c_i(u^{\max}_i)}{u^{\max}_i}.
\end{cases}
\end{align}
and for the linear $c_i(\cdot)$:
\begin{align}\label{eq:cases-linear4}
u_i(t, \hat{\beta}(T))=
\begin{cases}
u^{max}_i(t), \quad &\text{if} \quad h_i(t) > \hat{\beta}(T) v_i,\\
0, \quad &\text{if} \quad  h_i(t)  < \hat{\beta}(T) v_i,\\
\text{?},\quad &\text{if} \quad h_i(t)  = \hat{\beta}(T) v_i.
\end{cases}
\end{align}
One can see that in both cases, if $\hat{\beta_1}(T)> \hat{\beta_2}(T)\geq 0$, $u_i(t, \hat{\beta_2}(T))\geq u_i(t, \hat{\beta_1}(T))$ for all $i$ and all $t$. This, along with Assumption \eqref{ass:increasing}, leads to $c_i\big(u_i(t, \hat{\beta_2}(T))\big)\geq c_i\big(u_i(t, \hat{\beta_1}(T))\big)$ for all $i$ and all $t$, culminating in:

\begin{align}\label{eq:equation10}
\int_0^T \sum_{i=1}^m c_i\big(u_i(t, &\hat{\beta_2}(T))\big)\,dt \geq \int_0^T \sum_{i=1}^m c_i\big(u_i(t, \hat{\beta_1}(T))\big)\,dt.
\end{align} 

As a corollary, \eqref{eq:equation10} holds with equality if and only if $\overrightarrow{u}(t, \hat{\beta_2}(T)) = \overrightarrow{u}(t, \hat{\beta_1}(T))$ for all $t$ (excluding any switching points). Thus, if 
\begin{align}\label{eq:checking}
\int_0^T \sum_{i=1}^m c_i\big(u_i(t, \hat{\beta}(T))\big)\,dt=r,
\end{align} 
then $\overrightarrow{u}(t, \hat{\beta}(T)) = \overrightarrow{u^*}(t)$ also for all $t$. Therefore, we have the following result:
\begin{Proposition} \label{thm:linear}
For the case of separable, linear objective functions, i.e, $J(\overrightarrow{x}(T))=\inp*{\overrightarrow{p}}{\overrightarrow{x}(T)}$, the unique optimal control can be explicitly calculated using a number of evaluations of \eqref{eq:checking} that is logarithmic in the range of considered $\hat{\beta}(T)$'s.\footnote{This proposition does not apply for the general $J(\overrightarrow{x}(T))$, as the equivalent definition of $h_i(t)$ in \eqref{eq:equation10} would have to replace $\overrightarrow{p}$ with $[\partial J(\overrightarrow{x})/\partial \overrightarrow{x}]_{ \overrightarrow{x}= \overrightarrow{x}^*(T)}$, which can only be evaluated with knowledge of the optimal terminal opinion vector $\overrightarrow{x}^*(T)$.}
\end{Proposition}

Using the process outlined above, we can use a simple bisection algorithm to find $\beta^*(T)$ and to solve the optimal control problem using a \emph{single-shooting} approach\hide{(which is akin to Newton's method of finding a zero of a function \cite[page 484]{boyd2004convex})}. $\hat{\beta}(T)$ is adjusted so as to find the root of $\int_0^T \sum_{i=1}^m c_i\big(u_i(t, \hat{\beta}(T))\big)\,dt=r$. This significantly decreases the complexity of calculating the optimal control, since instead of evaluating and comparing potential optimal solutions that fulfill the necessary conditions in Theorem \ref{thm:crossing}, one can simply evaluate $\int_0^T \sum_{i=1}^m c_i\big(u_i(t, \hat{\beta}(T))\big)\,dt$ using \eqref{eq:cases-concave3} and \eqref{eq:cases-linear4} over a number of iterations that is logarithmic in the range of $\hat{\beta}(T)$ under consideration to explicitly characterize the unique optimal control.

The procedure outlined above is also instructive in understanding the relative importance of different channels at different times graphically. In particular, we will be interested in comparing ${u^{\max}_i h_i(t)}/{c_i(u^{\max}_i)}$ for concave $c(\cdot)$ and ${h_i(t)}/{v_i}$ for linear $c(\cdot)$ with $\hat{\beta}(T)$ (as in \eqref{eq:cases-concave3} and \eqref{eq:cases-linear4}). One can think of the terms containing $h_i(t)$ as a topographic relief map, signifying hills and valleys. $\hat{\beta}(T)$ represents a water-line, below which the valleys are flooded. The budget expenditure in this case is a monotone function of the area above water (see Fig. \ref{fig:waterfilling}). Therefore, the algorithm outlined is equivalent to adjusting the water-line so that the budget expenditure (evaluated as a function of the land above water) matches the budget constraint.
\begin{figure}[htb]
\begin{center}
\includegraphics[scale=0.75]{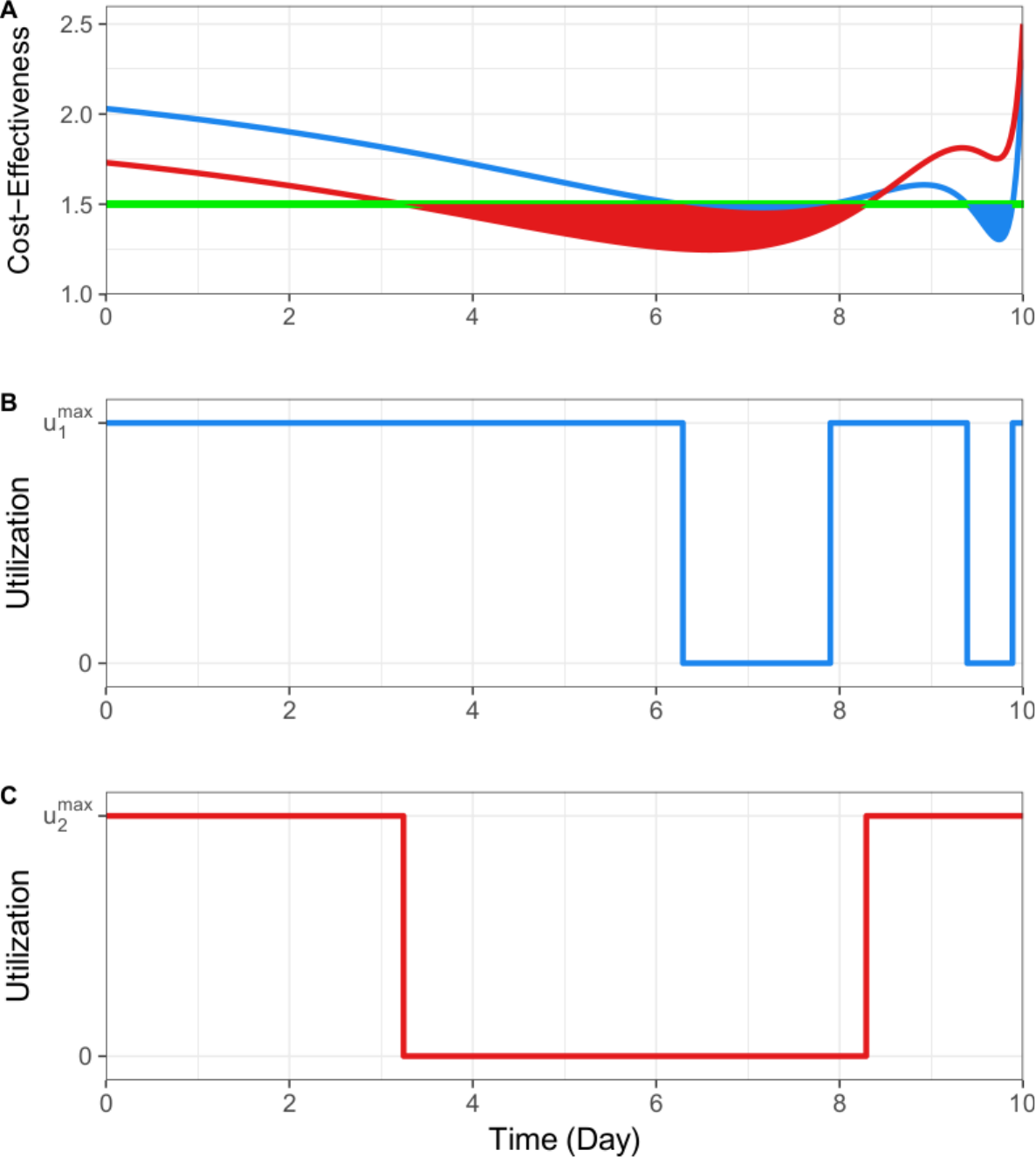}
\end{center}
\vspace{-0.1in}
\caption{\footnotesize{We demonstrate a case with two ${h_i(t)}/{v_i}$ functions for the case of linear $c_i(\cdot)$. Areas above the water-line ($\hat{\beta}(T)$) translate to $u_i(t, \hat{\beta}(T))=u^{max}_i(t)$, while those below translate to $u_i(t, \hat{\beta}(T))=0$. The amount of budget spent for this $\hat{\beta}(T)$ can thus be calculated from the resulting $u_i(t, \hat{\beta}(T))$, and so $\hat{\beta}(T)$ can be adjusted to find $\beta^*(T)$.}}
\label{fig:waterfilling}
\vspace{-0.1in}
\end{figure}

{Furthermore, the water-filling procedure shows the relative importance of  channels over time with respect to external influence}. As the optimal water-level is a monotone decreasing function of the budget available, one can see that the peaks in ${u^{\max}_i h_i(t)}/{c_i(u^{\max}_i)}$ and ${h_i(t)}/{v_i}$ signify the time intervals and channels that would be prioritized when the budget is tight, while if the budget is increased, more and more channels will be utilized at an increasing set of intervals. Therefore, we can consider the explicitly computable result of ${u^{\max}_i h_i(t)}/{c_i(u^{\max}_i)}$ (for the linear cost case, ${h_i(t)}/{v_i}$) to be a  \emph{direct metric}/ {total order} for the \emph{effect of advertising on channel $i$ at time $t$ on the outcome of the election}, which we shall henceforth call \textbf{cost-effectiveness} of a channel.\hide{\footnote{This persists in the presence of any type of noise in the the state dynamics.}}

One can extract some more insight from the structure of this metric to compare the relative importance of channels by considering \eqref{eq:weights} at extreme values of $t$:
{\begin{Remark}\label{Rem:gg}
 If $t\ll T$ (i.e., early on in the time horizon) and $T\gg \frac{1}{\xi_2}$, the deciding factor in comparing the cost-effectiveness of channels is their \textbf{total reach} (e.g., $\sum_{j=1}^n b_{ji}$ for channel $i$) per unit cost; for example, for the linear $c(\cdot)$ case:
\begin{align}\label{eq:reach}
\frac{h_i(t)}{v_i}\approx \left(\frac{1}{n}\sum_{j=1}^n p_j\right)\left(\frac{\sum_{j=1}^n b_{ji}}{v_i}\right),
\end{align}
as $\frac{1}{n}\sum_{j=1}^n p_j$ is the same for all channels.
\end{Remark}}
{\begin{proof}
From \S \ref{subsec:pre}, we know that $\xi_1=0<\xi_2\leq\ldots\leq \xi_n$, so
when $t\ll T$ and $T\gg \frac{1}{\xi_2}$, then $e^{-\xi_{j}(T-t)}\approx 0$ for $j\geq 2$ and 
$e^{-\xi_{1}(T-t)}= e^0=1$. Replacing these values, and $\overrightarrow{Q}(:,1)=\frac{1}{\sqrt{n}}\overrightarrow{1}_{n}$, in \eqref{eq:weights} completes the argument.
\end{proof}
}

\begin{Remark}\label{Rem:ll}
However, if $1-\frac{t}{T} \ll \frac{1}{\xi_{n}. T}$ (i.e., late on in the time horizon), \textbf{targeting} (e.g., how well a channel is aligned with the \emph{a priori} likelihood of people to vote) is more important than total reach; for example, for the linear $c(\cdot)$ case:
\begin{align}\label{eq:targetting}
\frac{h_i(t)}{v_i}\approx  \frac{\inp*{\overrightarrow{p}}{\overrightarrow{B}(:,i)}}{v_i},
\end{align}
\end{Remark}
{\begin{proof}
When $1-\frac{t}{T} \ll \frac{1}{\xi_{n}. T}$, then $e^{-\xi_{j}(T-t)}\approx 1$ for all $j \in [n]$. Replacing these values in \eqref{eq:weights} results in:
\begin{align*}h_i(t)&=\sum_{j=1}^N\inp*{\overrightarrow{Q}(:,j)}{\overrightarrow{p}}\inp*{\overrightarrow{Q}(:,j)}{\overrightarrow{B}(:,i)}=\inp*{\overrightarrow{p}}{\overrightarrow{B}(:,i)},
\end{align*}
due to the orthonormality of the eigenvectors in $\matr{Q}$ and the definition of an inner product.
\end{proof}
}

{This is instructive, as it shows that at the start of a campaign, cheap broadcast methods (that maximize total reach per unit cost) would be preferable to costly (premature) targeting of likely voters, while as election day approaches, the alignment of a channel with the likelihood of voting among its targets gradually increases in importance.}

\subsection{Separable Sigmoid Objective}\label{subsec:sigmoid}

{In this case, as shown in
\S Appendix \ref{subsec:proof1}, the equivalent $h_i(t)$ expression \eqref{eq:weights} will feature a term $\overrightarrow{\bLam^*}(T)$, instead of $\overrightarrow{p}$, that depends strongly on $|x^*_i(T)-\theta_i|$, how far agent $i$ is from changing their mind \eqref{obj:sigmoid}, for all $i$.} 
The further away $x^*_i(T)$ is from $\theta_i$ (i.e., the farther they are from changing their mind, or alternatively the more convinced they are), the smaller the relevant $\lambda^*_i(T)$. 
For a given $\epsilon\ll 1$, define the set of \textbf{late-deciders} \cite{hayes1996marketing} under the optimal advertising action $\overrightarrow{u^*}$ to be $\mathcal{L}:=\{j:~|x^*_j(T)-\theta_j|<\epsilon\}$. When $\mathcal{L} \neq \emptyset$ (i.e., there are late deciders), we can use the water-filling machinery in \S \ref{subsec:waterfilling} with the changes outlined below to approximate the cost-effectiveness of channels and to calculate the optimal allocation using the much faster method described therein.

We define $\overrightarrow{\bar{\bLam}}$ such that:
\begin{align*}
\bar{\lambda}_j=\begin{cases}0 ~~~& \text{for }j \notin \mathcal{L},\\
\frac{\alpha_j p_j}{2} ~~~&\text{for }j \in \mathcal{L}.\end{cases}
\end{align*}
Then, the approximate cost-effectiveness metric of channel $i$ with linear $c(\cdot)$ becomes:
\begin{align*}
\frac{h_i(t)}{v_i}= \frac{\inp*{\overrightarrow{Q}(:,j)}{\overrightarrow{\bar{\bLam}}}\inp*{\overrightarrow{Q}(:,j)}{\overrightarrow{B}(:,i)}e^{-\xi_{j}(T-t)}}{v_i}.
\end{align*}
This confirms the practical intuition that identifying the people who will decide late early in the campaign can delineate the whole trajectory of the campaign.

\section{Simulation Studies}\label{sec:simulations}
In this section, we first study a simple example to show that even in small networks, the optimal budget allocation across channels can have complicated, sometimes counter-intuitive, structures. Furthermore, we show that in many cases, the bound derived from Theorem \ref{thm:crossing} grows much slower than the number of agents, $n$. {Then, we study the performance of our algorithm on a real network derived from political discussions between MIT students prior to the 2008 US general election, and compare it to policies that use more simple centrality metrics that do not consider the temporal degrees of freedom of advertising policies.}

We first examine a network of 7 agents with linear objectives, with $\overrightarrow{p}=(3\%, 2\%, 10\%, 100\%, 6\%, 7\%, 1\%)$. Note that under these conditions, agent 4 is the only reliable voter, with all other agents having small probabilities of voting. The connections within the network are represented in Fig. \ref{fig:topology}; the off-diagonal elements of the Laplacian $\matr{L}$ are such that $l_{ij}=1$ if there is an edge in the figure between nodes $i$ and $j$ and zero otherwise.
Assume two equal (linear) cost channels are available to the advertiser: Channel 1, $\overrightarrow{B}(:,1) = (1,-1, 1, 0, 1, 0, 0)^T$ has a positive impact on agents 1, 3, and 5, but a negative impact on agent 2. It has no effect on the likely voter, agent 4. In contrast, channel 2, $\overrightarrow{B}(:,2) = (-1,1, 0, 1, 0, 0, 0)^T$, has a positive effect on the likely voter, but it has more limited effects on the rest of the agents. We solve the optimal buget allocation problem in Fig. \ref{fig:waterfilling1} using the waterfilling methodology of \S \ref{subsec:waterfilling}. As noted in Remarks \ref{Rem:gg} and \ref{Rem:ll}, at times $t\ll T=10$, the cost-effectiveness of the two equal cost channels is measured by \eqref{eq:reach}
which is larger for channel 1, even though it does not do a good job of targeting the likely voter. However, for $t$ close to $T$, we can see that the cost-effectiveness ranking  depends on the match between the reach of the channel and the likelihood of agents to vote \eqref{eq:targetting}, and therefore the cost-effectiveness of channel 2 is higher. This optimal control is bang-bang with bounded numbers of transitions, as proven in Theorem \ref{thm:crossing}. 

\begin{figure}[htb]
\begin{center}
\includegraphics[scale=0.4]{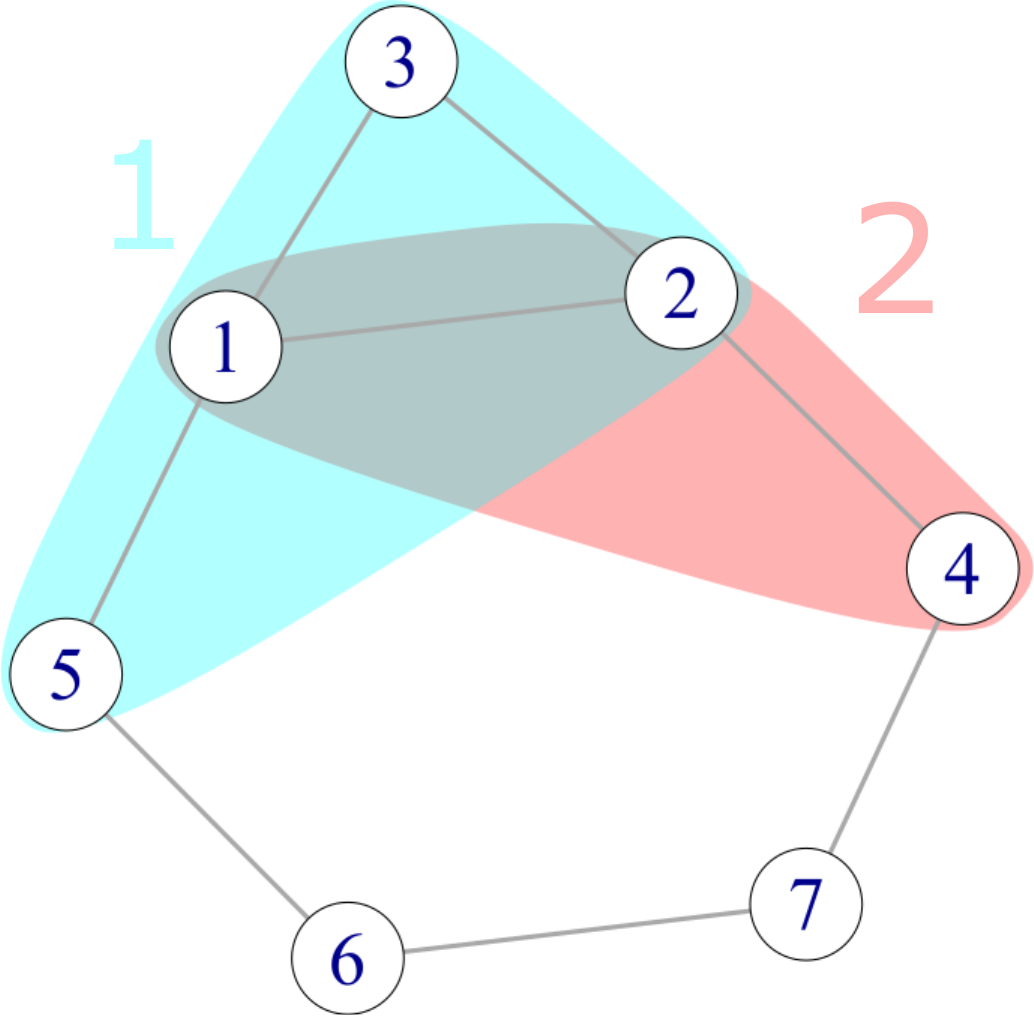}
\end{center}
\vspace{-0.1in}
\caption{\footnotesize{A network of $n=7$ agents with $m=2$ influence channels. The lines in solid black represent the underlying communciation network $\matr{L}$. The blue and red boxes delineate the two channels that are available for influence in terms of agents affected (but not intensity).}}\label{fig:topology}
\vspace{-0.1in}
\end{figure}

\begin{figure}[htb]
\begin{center}
\includegraphics[scale=0.65]{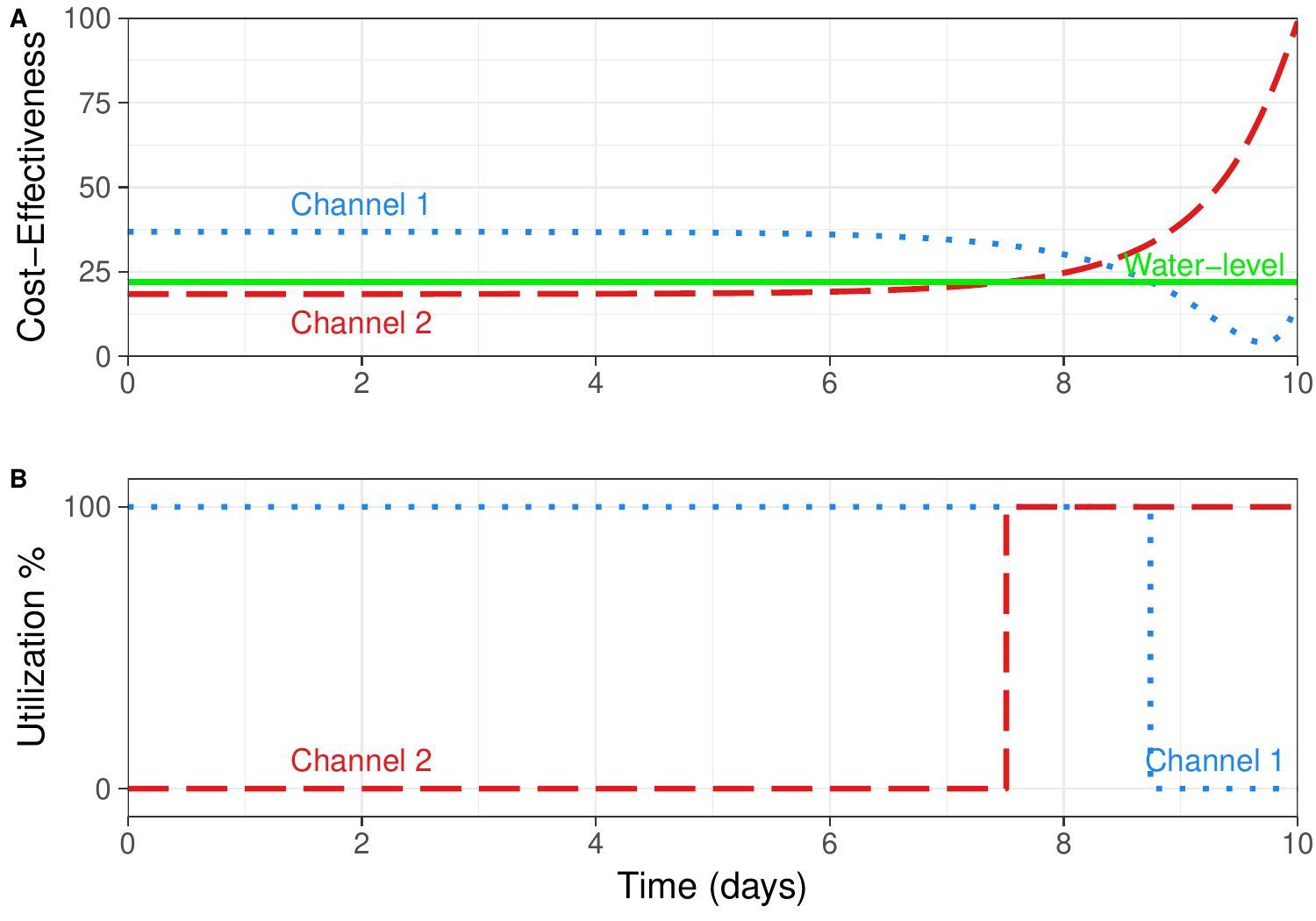}
\end{center}
\vspace{-0.2in}
\caption{\footnotesize{(A) We plot the cost-effectiveness of the two channels over a time horizon of $T=10$ days when they have equal cost ($c_i(u_i)= v_i u_i$ for $i=1,2$ with $v_1=v_2=1$). We then derive the water-level $\hat{\beta}$ for $r=11$ (in green). (B)  The optimal water-level determines the optimal utilization rate of the two channels at different times. As can be seen, the channel with the most reach (channel 1) is prioritized at small $t$, and the one that is most aligned with the likelihood to vote (channel 2) is prioritized late as the election draws near.}}
\label{fig:waterfilling1}
\vspace{-0.2in}
\end{figure}

One important question, especially from a computational point of view, is how tight the upper-bounds on the number of switches are. The most general bound (Theorem \ref{thm:crossing}.2.a) grows with the number of agents in the system, potentially leading to a large computational burden. On the other hand, knowing $\bLam^*(T)$ will allow us to use tighter bounds, like that in Theorem \ref{thm:crossing}.2.b. We simulated 1000 random connected Erdos-Renyi graphs with uniformly random linear objective functions for the case of $\overrightarrow{B}(:,1)=(1, \overrightarrow{0}_{1\times (n-1)})^T$, and plotted the mean, variance, and maximum value of the bound in Theorem \ref{thm:crossing}.2.b as the number of agents was varied. As can be seen in Fig. \ref{fig:bound}, this latter bound is much smaller (around 10 for 200 agents), and its growth with respect to the number of agents is very slow. This is significant since, from an applied perspective, the advertiser {can enumerate and evaluate a much smaller set of candidate optimal solutions, and yet can be reasonably sure that the best such policy is globally optimal.}

\begin{figure}[htb]
\begin{center}
\includegraphics[scale=0.85]{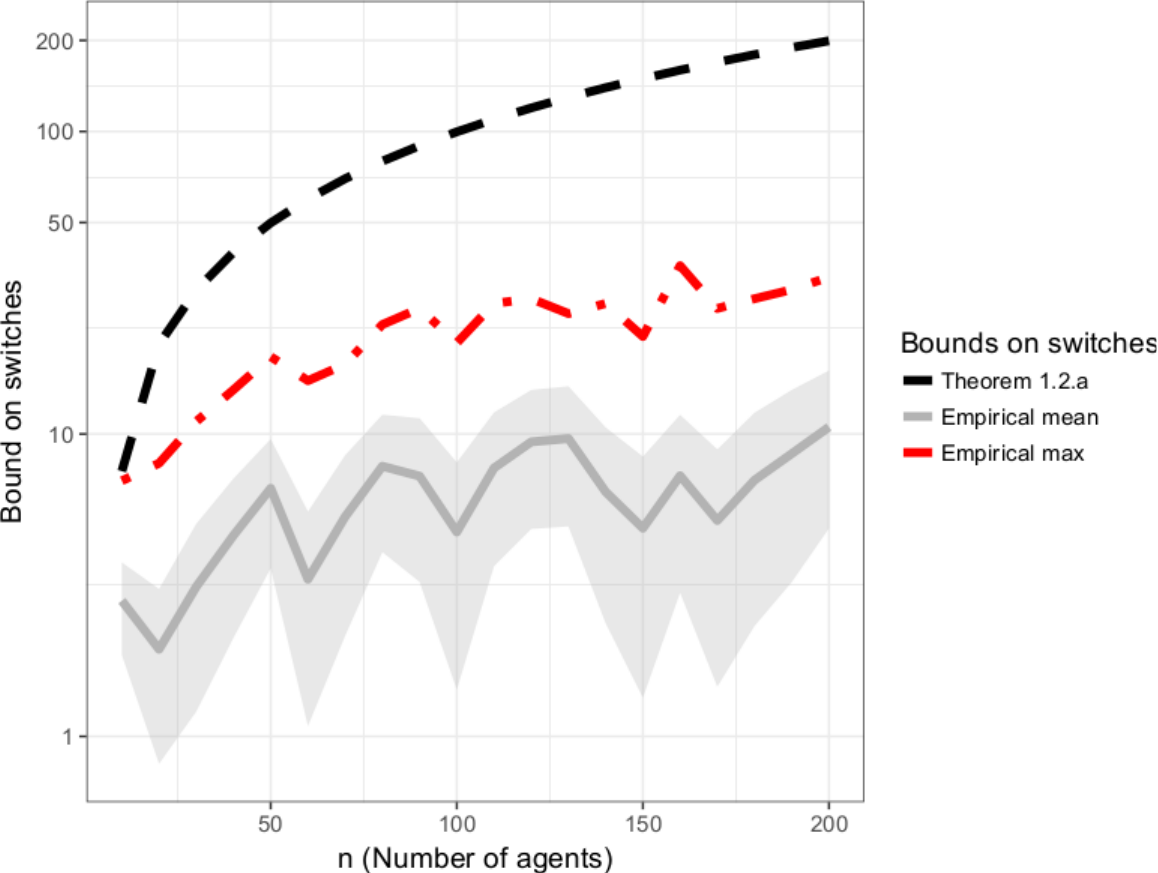}
\end{center}
\vspace{-0.1in}
\caption{\footnotesize{We plot the upper-bounds on the number of switches of the optimal resource allocation derived from Theorem \ref{thm:crossing} for 1000 random Erdos-Renyi graphs for a channel that only affects the first agent as the size of the network is varied. The dashed bound is from Theorem \ref{thm:crossing}.2.a, and can be seen to grow with the number of agents. For a linear objective (Theorem \ref{thm:crossing}.2.b), the grey line (with the related standard deviation band) shows the mean bound on the number of switches, while the dashed red line shows the empirical maximum of the bound over 1000 runs. We can see that both these values are significantly smaller than the bound from Theorem \ref{thm:crossing}.2.a and increase at a much slower rate with the size of the network.}}
\label{fig:bound}
\vspace{-0.1in}
\end{figure}

We now study the performance of our algorithm on a test scenario derived from the MIT Social Evolution data-set \cite{madan2012sensing}. In this data-set, among other data, the political opinions and communication patterns of 84 MIT students are recorded in the period prior to, and following, the 2008 US presidential election. Furthermore, the living sector and year of the students was recorded. 
{We consider the problem of deciding how the campaign of Barack Obama should have invested its resources to disseminate campaign literature in order to guarantee the best electoral outcome.} While this is admittedly a stylized and somewhat simplistic, it adequately demonstrates how the model could be specified and identified.

In particular, we focus on a social network derived from the reported political discussions between students conducted on 2008-09-09 and 2008-10-19, the only two surveys conducted before the November 4th election. We consider ``discussion'' to be an undirected communication between individuals, and thus we aggregate communications that are flagged by both participants. However, we sum distinct communications between two individuals to denote a stronger bond. This process is used to generate the Laplacian communication matrix $\matr{L}$, including by normalizing discussions by time-frame $T=66$ days. The discussion graph is plotted in Fig. \ref{fig:discussions}.

\begin{figure*}[htb]
\begin{center}
\includegraphics[scale=0.34]{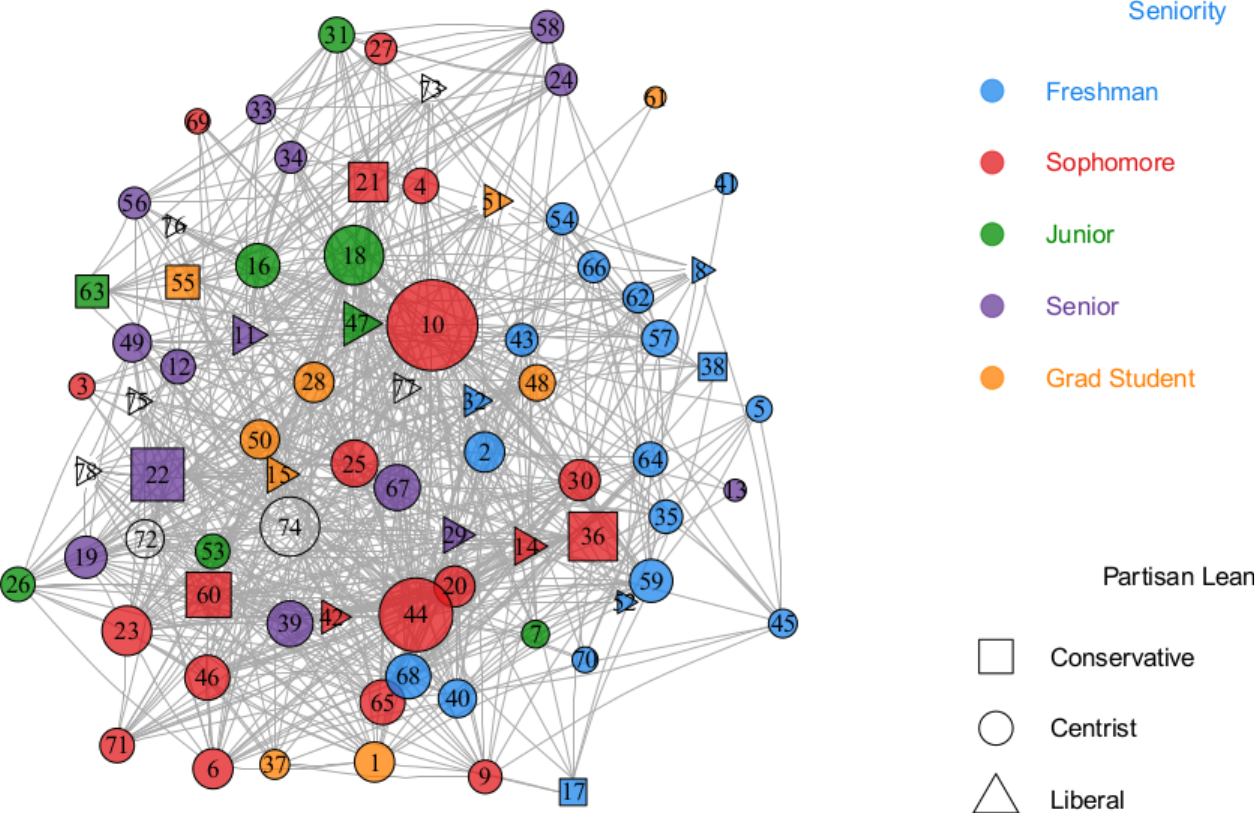}
\hfill
\includegraphics[scale=0.34]{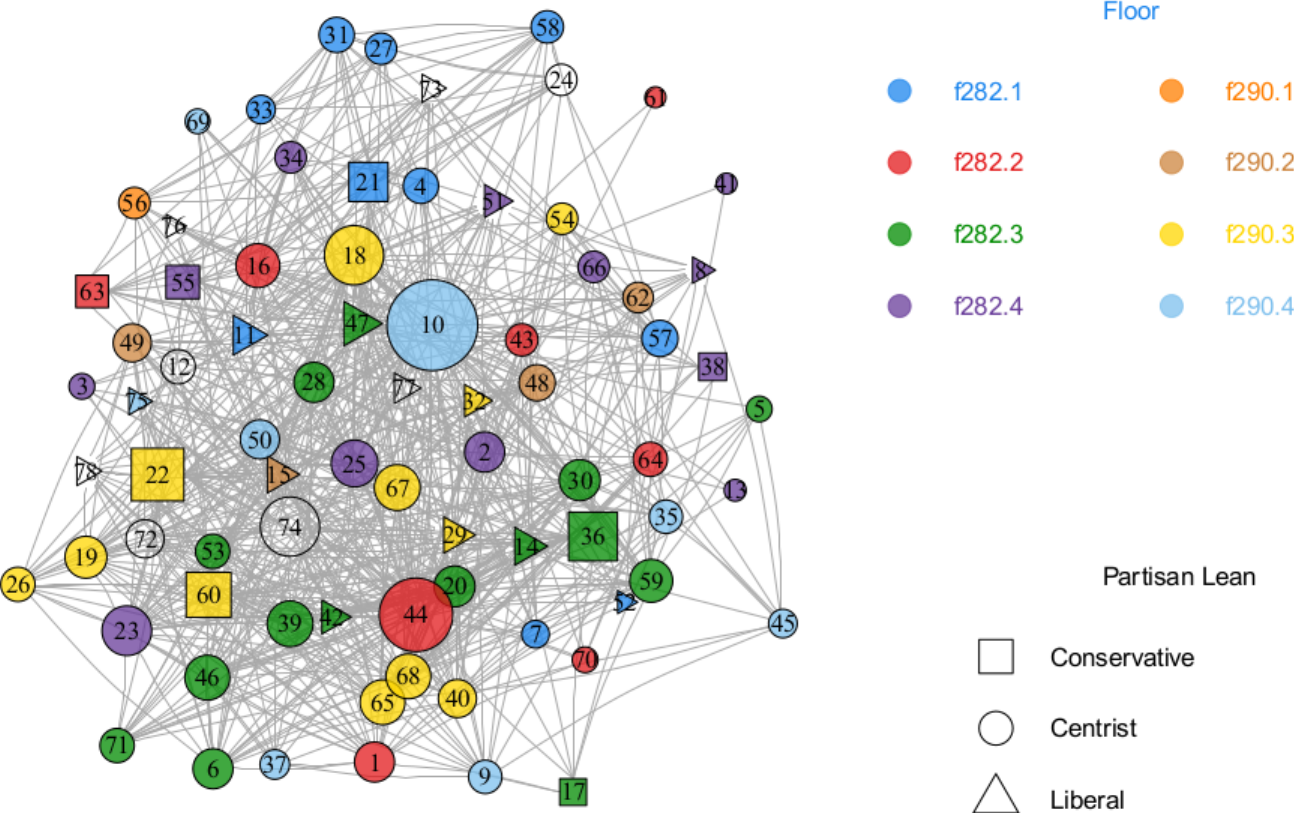}
\\\vspace{0.1in}
\includegraphics[scale=0.34]{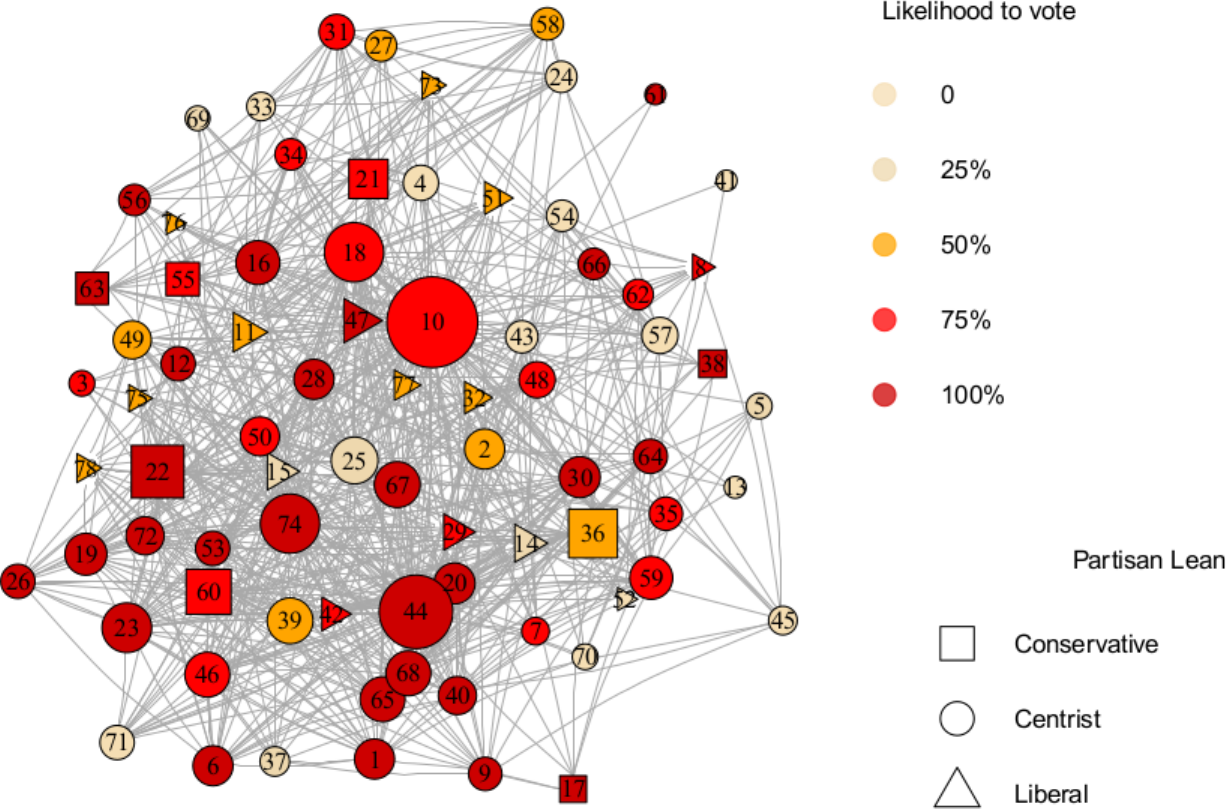}
\hfill
\includegraphics[scale=0.34]{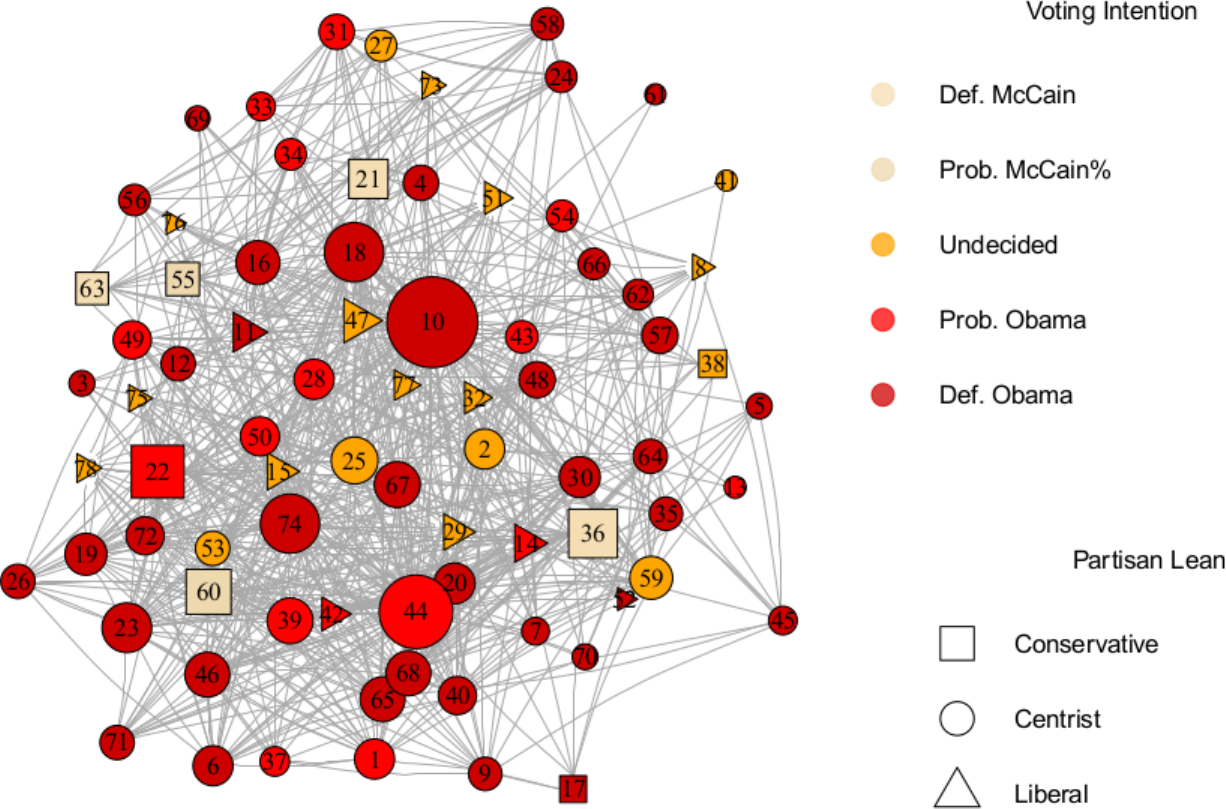}
\end{center}
  \vspace{-1em}
\caption{\footnotesize{In this figure, political discussions between students are mapped as a graph, with the weight of links being derived from the frequency of discussions. In the left-hand graph, nodes are colored according to the seniority of the students, while in the right-hand graph, they are colored in according to their residence, which are the two determinants of advertising channels.}}\label{fig:discussions}
\vspace{-0.2in}
\end{figure*}

We then calculate the channel matrix $\matr{B}$ for channels that represent the 8 dorm floors and 5 seniority levels (freshman, sophomore, junior, senior, graduate), leading to a total $m$ of 13. The channels mapping to dorm floors capture possible advertising on bulletin boards, for example, while the 5 seniority-based channels could represent e-mail lists targeting specific graduating years.
The weight of the effect of each channel is derived from the self-reported liberal or conservative initial bias of the individuals, as the effect of advertising depends on its alignment with the values of the target \cite{atkin1973quality}. In this example, we consider the propagation of campaign literature targeted at liberals, and thus off-putting to conservatives. This can model any of the wedge issues of the campaign (e.g., the Iraq war \cite{cnn2008iraq}). Thus, advertising can have a negative effect on outcomes for the campaign, making some individuals less likely to vote for the candidate. \red{Thus, we assign a non-zero value to $b_{ik}$ if $i$ is in the $k$-th dorm floor/seniority group, with the sign being determined by individual $i$'s self-described ``liberal'' (from the perspective of the Obama campaign, positive) or ``conservative'' (respectively, negative) affiliation, and the magnitude being determined by their self-described strength of identification with that affiliation \footnote{\red{The mapping for $b_{ik}$ within channel $k $ was as follows: ``Extremely conservative'':$-1$, `` Conservative'':$-0.66$, ``Slightly conservative'':$-0.33$, ``Moderate middle of the road'': $0$, ``Slightly liberal'':$0.33$, ``Liberal'':$0.66$,
``Extremely liberal'':$1$}.}}

 Furthermore, and again for simplicity, we assume that all the channels have a similar linear cost, $\overrightarrow{v}=\overrightarrow{1}$, and have similar small effects on the voting intentions of participants $\overrightarrow{u}_{max} = 0.01$. The channels are also shown in Figure \ref{fig:discussions}.

We consider a linear objective for the campaign. While the more complex sigmoid objective functions are a better model for decision-making, we consider the simpler linear case for tractability. We map the self-reported likelihood of voting of participants in September '08 to a $[0,1]$ scale, taking 5 equally spaced values, and constituting the vector $\overrightarrow{p}$.  While self-reported turnout has been shown to be an unreliable predictor of voting behavior \cite{granberg1991self}, we operate under the reasonable assumption that more reliable information is not available to the political 
campaign. These likelihoods of voting can be seen in Fig. \ref{fig:discussions}.

Finally, we instantiate the opinions of individuals $\overrightarrow{x}(0)$ (representing their voting intentions, as viewed by the Obama campaign) with the self-reported voting intention of individuals in September '08, which takes 8 values, mapped to values between $[-1,1]$. Again, for simplicity, we pool third-party voters with undecided voters. A more realistic scenario with vectors of opinions would be able to more accurately capture the diversity in opinions, but would not be as instructive as the current example for the performance of our policy and the resulting centralities. These initial preferences can be seen in Fig. \ref{fig:discussions}.

\begin{figure}[h]
\begin{center}

\includegraphics[scale=0.6]{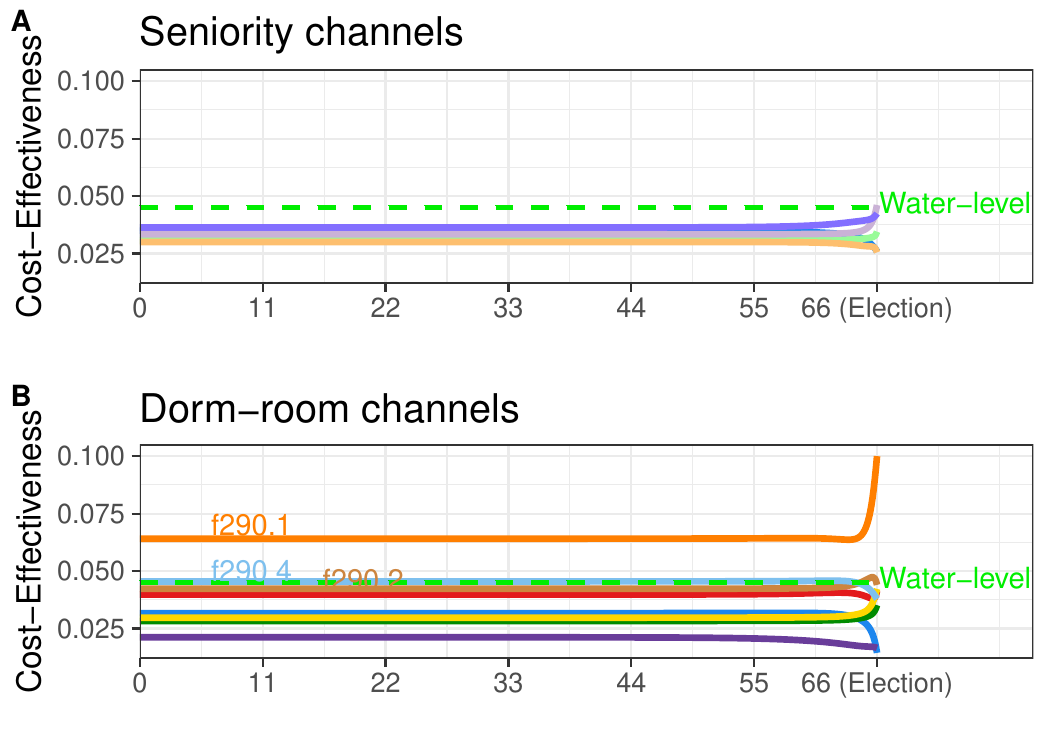}

\includegraphics[scale=0.6]{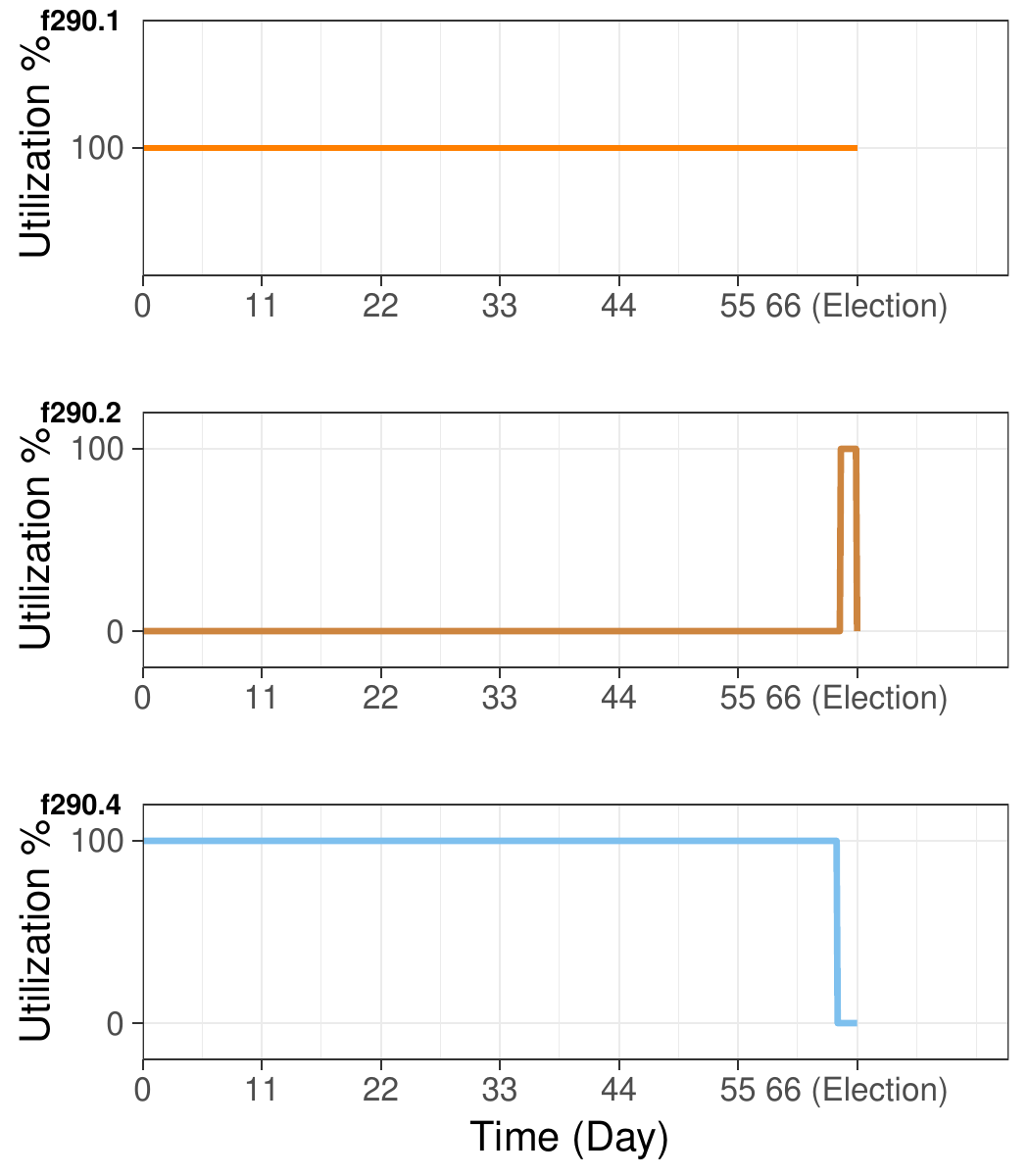}
\end{center}
\vspace{-0.1in}
\caption{\footnotesize{We demonstrate a case with two ${h_i(t)}/{v_i}$ functions for the case of linear $c_i(\cdot)$. Areas above the water-line ($\hat{\beta}(T)$) translate to $u_i(t, \hat{\beta}(T))=u^{max}_i(t)$, while those below translate to $u_i(t, \hat{\beta}(T))=0$. The amount of budget spent for this $\hat{\beta}(T)$ can thus be calculated from the resulting $u_i(t, \hat{\beta}(T))$, and so $\hat{\beta}(T)$ can be adjusted to find $\beta^*(T)$.}}
\label{fig:waterfilling_mit}
\vspace{-0.1in}
\end{figure}

In Fig. \ref{fig:waterfilling_mit}, we show the water-fillling procedure and the resulting optimal utilization of the channels for this problem for a budget of $r=52$. We observe that the optimal budget allocation only uses three of the dorm floor channels, with time-variations in the use of channels f290.2 and f290.4. This is somewhat counter-intuitive given the significantly higher reach of the seniority-based channels, which are unused by the optimal allocation, while dorm floor f290.1, which only includes one solitary individual, is used throughout the time period at the maximum possible rate. However, as the cost of utilizing a channel is taken to be proportional to its reach, f290.1 is utilized because it is very effective \emph{relative to its cost.}  

To benchmark our results, we compared the results of the optimal budget allocation policy on electoral outcomes to policies based on different types of static centralities: between-ness centrality, 
eigen-centrality, 
Page-rank, and degree centrality. For the comparison, we ranked channels according to a channel-weighted sum of the centrality in question (to account for possible negative effects of a channel on an individual), and allocated our budget to the highest-ranked channels at the maximum possible static rate until the exhaustion of our budget ($r=52$). Table \ref{tab:centralities} summarizes the relative differences in outcomes between the optimal dynamic budget allocation policy and the static benchmarks. We see that using our optimal water-filling algorithm based on our novel cost-effectiveness metric leads to a 26\% increase in the expected number of votes compared to the best static policy based on common centrality measures, a significant improvement. 

\red{Note that} \blue{due to our results outlined in Section \ref{subsec:waterfilling},} {the total budget does not affect the relative priority assigned to the channels by either the optimal algorithm (as it does not change the cost-effectiveness of channels) or the heuristics. Rather, it determines how many of the channels with high priorities can be used without going over the budget. The results presented in Table \ref{tab:centralities} are representative in the regime where the budget is a binding constraint in the choice of advertising channels.}

\begin{table}[htb]
\begin{center}
\begin{tabular}{ | c | c | }
\hline
Basis of decision-making & Expected number of votes\\
  \hline
  \hline
  \textbf{Optimal Allocation} & \textbf{39.03}\\
  \hline
Betweenness centrality & 29.86\\
\hline
Eigen-centrality  & 29.16\\
\hline
Page-rank & 29.89\\
\hline
Degree centrality & 30.03\\ \hline  
\end{tabular}
\end{center}
\caption{Expected number of votes gained based on the greedy allocation of budget to the highest-ranked channels according to different types of centralities.}\label{tab:centralities}
\vspace{-2.5em}
\end{table}

\section{Extension to Asymmetric Interaction Weight Matrices $\matr{A}$}\label{sec:directed}

\green{We assumed, in \S \ref{subsec:model} that the interaction matrix $\matr{A}$ was symmetric, i.e., $a_{ij}=a_{ji}$ for all pairs $i, j \in [n]$.
However, asymmetric interactions can also be considered in our framework. We consider the weighted digraph $G=(V,E,\matr{A})$, where $(i,j)\in E$ represents a link from $j$ to $i$ (i.e., the associated weight $a_{ij}$ is the weight agent $i$ assigns to the opinion of agent $j$ in its opinion update).

We first need a definition akin to our assumption of connectedness for the undirected graph:

\begin{Definition}
A \emph{quasi-strongly connected (QSC)} digraph is a digraph that contains a rooted out-branching (i.e., a directed spanning tree) as a sub-graph \cite[p. 51, Definition 3.7]{mesbahi2010graph}, \cite[Definition 6]{proskurnikov2017tutorial}
\end{Definition}

We know, from Gershgorin's Disc Theorem \cite[p. 357, Theorem 7.2.1]{golub2012matrix}, that the eigenvalues of an asymmetric Laplacian matrix $\matr{L}$ will have non-negative real parts. Note also that due to the row-sums being equal to zero, the smallest eigenvalue  of $\matr{L}$ is always \textbf{zero}, its multiplicity is at least one , and one of its associated eigenvectors is $\frac{1}{\sqrt{n}}\overrightarrow{1}_{n}$. For quasi-strongly connected digraphs we know that this eigenvalue is simple \cite[p. 51, Proposition 3.8]{mesbahi2010graph}, \cite[Lemma 8]{proskurnikov2017tutorial}.

As mentioned in Remark \ref{rem:asymmetric}, \emph{Lemma \ref{lem:existence} and Theorem \ref{thm:crossing}.\ref{thm:first} carry over, with the number of switches being bounded above by $n-1$, when the interaction matrix is asymmetric, under the conditions that:}
\begin{enumerate}
\item the digraph $G$ is quasi-strongly connected (QSC).

\item all the eigenvalues of the asymmetric Laplacian $\matr{L}$ are strictly real;

\item $c_i(\cdot)$ is linear.

\end{enumerate}
Furthermore, the arguments around water-filling presented in \S \ref{subsec:waterfilling} also follow, with a different definition of the function $h_i(t)$ for all $i \in [m]$. 

\begin{Remark}
Note that every weighted Directed Acyclic Graph (DAG) has a Laplacian with real eigenvalues. 
\end{Remark}

\begin{proof}
Every DAG $G=(V,E)$ has a topological ordering, in which its vertices are ordered such that for all $(u,v)\in E$, $u$ appears before $v$ in the topological order.\cite[p. 151]{gross2013handbook} Given this ordering, the asymmetric Laplacian $\matr{L}$ will be upper-triangular, and thus the elements of its diagonal, which are real, will be its eigenvalues.
\end{proof}

Here, we document the changes in the proof process for the case of asymmetric interactions.

\subsection{Changes to \S \ref{subsec:pre}} 
We assume that the underlying digraph $G$ is quasi-strongly connected (QSC).
We now also consider the case where the geometric multiplicities of the eigenvalues may be less than their algebraic multiplicities. By our assumption, all these eigenvalues are real, and thus so are the associated simple and generalized eigenvectors.\footnote{This is because the generalized eigenvectors associated with eigenvalue $\xi_i$ create a basis for the null-space of $(\matr{L}-\xi_i \matr{I})$ in $\mathbb{R}^n$.} Furthermore, the set of (generalized) eigenvectors can be chosen to be linearly independent, creating a canonical basis for $\mathbb{R}^n$. Also let $\matr{Q}$ be a matrix whose columns are these generalized eigenvectors of $\matr{L}$, with the generalized eigenvectors corresponding to smaller eigenvalues appearing first, those corresponding to the same eigenvalue and generated from the same Jordan chain appearing one after the other in order of the chain (starting with the simple eigenvector), and chains corresponding to the same eigenvalue being ordered according to the size of the chain (largest first). Then, we will have $\matr{L}=\matr{Q} \matr{\Xi} \matr{Q}^{-1}$, where now $\matr{\Xi}$ is the upper-triangular Jordan normal form of $\matr{L}$, with real, non-negative eigenvalues on the diagonal, with the eigenvalues appearing from the smallest to the largest. Let $\mu\leq n$ be the total number of Jordan blocks in $\matr{\Xi}$; let $\matr{\Xi_i}$ be the $i-$th Jordan block, with size $|\Xi_i|$ and corresponding to eigenvalue $\xi_i$, for $i \in [\nu]$. Note that $\sum_{i=1}^\nu|\Xi_i|=n$; we also set $|X_0|=0$. While not all $\xi_i$ may be distinct, yet we have $\xi_i \leq \xi_{i+1}$ for $i \in [\nu-1]$ according to our construction.
Note also that $\matr{Q}(:,1)=\frac{1}{\sqrt{n}}\overrightarrow{1}_{n}$ for the simple $0$ eigenvalue. 

\subsection{Changes to Lemma \ref{lem:existence}} 

The statement and proof of Lemma \ref{lem:existence} do not change for this case.

\subsection{Changes to the proof of Theorem \ref{thm:crossing} } 

 Nothing changes until the arguments around the zeros of $\inp*{\overrightarrow{\bLam^*}(t)}{\overrightarrow{B}(:,i)}=\text{constant}$ in \S Appendix \ref{subsec:proof1}. 


Now, we define $\overrightarrow{y}(t):=\matr{Q}^{-1}\overrightarrow{\lambda^*}(t)$, and we will again have:
\begin{align*}
\dot{\overrightarrow{y}}&=  \matr{\Xi} \overrightarrow{y}.
\end{align*}
However, now due to the Jordan normal form structure of $\matr{\Xi}$, the solution of the ODE will change; for all $i\in [n]$ and $k\in[\nu]\Cup \{0\}$ such that $\sum_{j=0}^k |\Xi_j|<i\leq \sum_{j=0}^{k+1} |\Xi_j|$, first let $\ell_i:= i - \sum_{j=0}^k |\Xi_j|$ and $k_i=k$. Then, we will have:  
\begin{align*}
y_i(t)&= \big(\sum_{j=0}^{\ell_i} \frac{t^j}{j!} y_{i-j} (T)\big) e^{\xi_{k_i}(t-T)}= \sum_{j=0}^{\ell_i} \frac{t^j}{j!}\inp*{\overrightarrow{Q}^{-1}(i-j,:)}{\overrightarrow{\bLam^*}(T)}e^{-\xi_{k_i}(T-t)}.
\end{align*}
Therefore:
\begin{align*}
\overrightarrow{\bLam^*}(t)& = \matr{Q}\overrightarrow{y}= \matr{Q}\bigg[\sum_{j=0}^{\ell_i} \frac{t^j}{j!}\inp*{\overrightarrow{Q}^{-1}(i-j,:)}{\overrightarrow{\bLam^*}(T)}e^{-\xi_{k_i}(T-t)}\bigg]_{i} = \sum_{i=1}^n \bigg[\sum_{j=0}^{\ell_i} \frac{t^j}{j!}\inp*{\overrightarrow{Q}^{-1}(i-j,:)}{\overrightarrow{\bLam^*}(T)}e^{-\xi_{k_i}(T-t)}\bigg] \overrightarrow{Q}(:,i).
\end{align*}

And:
\begin{align}\label{eq:equation8}
\inp*{\overrightarrow{\bLam^*}(t)}{\overrightarrow{B}(:,i)}&= \sum_{j=1}^n \bigg[
\sum_{\varrho=0}^{\ell_j} \frac{t^\varrho}{\varrho!}\inp*{\overrightarrow{Q}^{-1}(j-\varrho,:)}{\overrightarrow{\bLam^*}(T)}e^{-\xi_{k_j}(T-t)}
\bigg] \bigg[\inp*{\overrightarrow{Q}(:,j)}{\overrightarrow{B}(:,i)}\bigg]\notag\\&=
 \sum_{j=1}^n \bigg[\inp*{\overrightarrow{Q}(:,j)}{\overrightarrow{B}(:,i)} \sum_{\varrho=0}^{\ell_j} \frac{t^\varrho}{\varrho!}\inp*{\overrightarrow{Q}^{-1}(j-\varrho,:)}{\overrightarrow{\bLam^*}(T)}\bigg]e^{-\xi_{k_j}(T-t)}=s.
\end{align}
Given that the smallest eigenvalue ($0$) is simple and associated with eigenvector $\frac{1}{\sqrt{n}}\overrightarrow{1}_{n}$, this is equivalent to:
 \begin{align}\label{eq:equation9}
      \sum_{j=2}^n \bigg[\inp*{\overrightarrow{Q}(:,j)}{\overrightarrow{B}(:,i)} \sum_{\varrho=0}^{\ell_j} \frac{t^\varrho}{\varrho!}\inp*{\overrightarrow{Q}^{-1}(j-\varrho,:)}{\overrightarrow{\bLam^*}(T)}\bigg]e^{-\xi_{k_j}(T-t)}+(\bigg[\sum_{h=1}^n b_{hi} \inp*{\overrightarrow{Q}^{-1}(1,:)}{\overrightarrow{\bLam^*}(T)}\bigg]-s)=0.
\end{align}
To find the zeros of \eqref{eq:equation9}, we again reason based on whether any of the coefficients of the exponentials in \eqref{eq:equation9} are non-zero or not.

\textbf{I)} In the case where for some $2\leq \iota \leq \nu$, \[\sum_{j\in \matr{\Xi_\iota}} \bigg[\inp*{\overrightarrow{Q}(:,j)}{\overrightarrow{B}(:,i)} \sum_{\varrho=0}^{\ell_j} \frac{t^\varrho}{\varrho!}\inp*{\overrightarrow{Q}^{-1}(j-\varrho,:)}{\overrightarrow{\bLam^*}(T)}\bigg]\neq 0,\] we appeal to a generalized form of Lemma \ref{lem:n-1} to complete the proof that \eqref{eq:equation9} has at most $n-1$ roots:

\begin{Lemma}\label{lem:N-1}
For $K>1$ and $M\leq K$, let $\eta_1>\eta_2>\cdots>\eta_M>0$. Also, for $i= 1,\cdots, M,$ and $j =1,\cdots,w_i$ such that $\sum_{i=1}^M w_i = K$, let $d_{i,j}$ be a real number. Then, if in the function $f : \mathbb{R} \to \mathbb{R}$, $f(t) = \sum_{i=1}^{M}\sum_{j=0}^{w_{i}-1} d_{i,j} t^j\eta_i^{t}$, we have $d_{i,{w_{i}-1}}\neq 0$ for all $i$, $f(t)$ has at most $K-1$ zeros.
\end{Lemma}

 The proof of this lemma is provided in \S Appendix \ref{sec:lemma_gen}. Therefore, the number of roots of \eqref{eq:equation9} is bounded above by the number of non-zero elements in $\big\{\sum_{j\in \matr{\Xi_\iota}} \inp*{\overrightarrow{Q}(:,j)}{\overrightarrow{B}(:,i)} \sum_{\varrho=0}^{\ell_j} \frac{t^\varrho}{\varrho!}\inp*{\overrightarrow{Q}^{-1}(j-\varrho,:)}{\overrightarrow{\bLam^*}(T)}\big\}_{\iota=2}^\nu$, which is itself bounded above by $n-1$.

 \textbf{II)} This again leaves the case where for  all $2 \leq \iota \leq \nu$, \[\sum_{j\in \matr{\Xi_\iota}} \bigg[\inp*{\overrightarrow{Q}(:,j)}{\overrightarrow{B}(:,i)} \sum_{\varrho=0}^{\ell_j} \frac{t^\varrho}{\varrho!}\inp*{\overrightarrow{Q}^{-1}(j-\varrho,:)}{\overrightarrow{\bLam^*}(T)}\bigg]\neq 0.\] 

\textbf{II-1)} If $\sum_{h=1}^n b_{hi} \inp*{\overrightarrow{Q}^{-1}(1,:)}{\overrightarrow{\bLam^*}(T)}\neq s$, then \eqref{eq:equation9} has no root and the claim holds.

\textbf{II-2)} Else the rest of the proof is the same as in \S Appendix \ref{sec:lemma2}, part II-2, for the linear $c_i(\cdot)$ case.

\subsection{Changes to \S \ref{subsec:waterfilling}}
Naturally, from the arguments above, Proposition \ref{thm:linear} and the arguments around the water-filling approach carry over for the case of linear $c_i(\cdot)$ and separable linear $J(\cdot)$, under the condition that \eqref{eq:weights} is replaced with:
\begin{align}\label{eq:weights2} h_i(t)&= \sum_{j=1}^n \bigg[\inp*{\overrightarrow{Q}(:,j)}{\overrightarrow{B}(:,i)} \sum_{\varrho=0}^{\ell_j} \frac{t^\varrho}{\varrho!}\inp*{\overrightarrow{Q}^{-1}(j-\varrho,:)}{\overrightarrow{p}(T)}\bigg]e^{-\xi_{k_j}(T-t)}
\end{align}

}
\section{Summary and Discussion}
We consider the problem of optimally allocating a finite budget over time across several advertising channels. We showed, using Pontryagin's maximum principle, that the optimal allocation follows a bang-bang structure, in which we either invest fully in a channel or not at all. In other words, to maximize the effectiveness of our budget, we should invest fully over a number of waves, and let the effect of the waves propagate in between waves. Furthermore, we show that the number of advertising waves during which we invest fully is, in practice, much smaller than the number of agents in the network. This result greatly facilitates the explicit computation of the optimal allocation policy over time. 
 Furthermore, we showed that the exact optimal control can be calculated using an efficient water-filling procedure for a linear objective. 
From this water-filling procedure, we rigorously defined ``cost-effectiveness'' as a metric for ranking and comparing the influence of different channels at differing times on outcomes. Finally, applying our results to the sigmoid approximation of the electoral campaign/voting model model confirmed the intuitive notion that identifying last-deciders determines the campaign strategy.

These results can be generalized in various ways. The notion of channel interaction in this work did not come with any constraints on the presence or attention of the channel members. Adding such a constraint can more clearly model real-world interactions. Furthermore, the linear model of influence is also a constraint that may be relaxed to obtain more general structures on influence control. Finally, this work looked at a single issue where each agents opinion was represented with a scalar - the same methodology can be extended to find optimal advertising strategies with vectors of opinions.

\ifCLASSOPTIONcompsoc
  \section*{Acknowledgments}
\else
  \section*{Acknowledgment}
\fi

The work of Eshghi, Zhao, D'Souza, and Swami was supported, in part, by the Army Research Laboratory Network
Science CTA under Cooperative Agreement W911NF-09-2-0053.
The work of Preciado was supported, in part, by the US National Science Foundation under grant CAREER-ECCS-1651433.

\ifCLASSOPTIONcaptionsoff
  \newpage
\fi

\bibliographystyle{IEEEtran}
{ \bibliography{biblio_2}}

\appendices

\section{Proof of Lemma \ref{lem:existence}}\label{sec:lemmaexist}
\begin{proof}
{The solutions to the dynamics ODEs \eqref{eq:vecdynamics}\rdout{,} \red{and} \eqref{eq:newstates}\rdout{, and \eqref{eq:newstates1}} exist for all $\overrightarrow{u}\in \mathcal{U}$ as the RHS terms are locally Lipschitz. We also have that for every $(t, \overrightarrow{x}(t), \gamma(t){, q(t)})$, the set:
\begin{align*}
\mathcal{P}:=\bigg\{\big((- \matr{L}\overrightarrow{x}(t) + &\matr{B}\overrightarrow{u}(t) + \overrightarrow{e}(t))^T, - \sum_{k=1}^m  c_k(u_k(t)),1\big)\bigg|\overrightarrow{u}_{m*1}(t): 0\leq u_k (t) \leq u_k^{\max}(t), k \in [m]\bigg\}
\end{align*}
 is compact, as the domain is compact and the functions are continuous. $\mathcal{P}$ is also convex, as in the first $n$ dimensions, it is mapped linearly from a convex set, in the $n+1$-th dimension it is also a linear mapping from a convex set (for linear $c_i(\cdot)$), and it is constant in the $n+2$-th dimension.
 Thus, according to Filippov's theorem \cite[page 119]{liberzon2012calculus}, the reachable set $R^{T}(\overrightarrow{x_0}, 0\mathstr{, 0}):=\{(\overrightarrow{x}(T), \gamma(T){, q(T)}): \overrightarrow{x}(0)= \overrightarrow{x_0}, \gamma(0)=0{, q(0)=0}, u\in \mathcal{U}\}$ from $(\overrightarrow{x_0}, 0{, 0})$ is compact and convex. The reachable set will still be compact and convex if we restrict it by only looking at controls that lead to $\gamma(T)\geq-r$ (as it is an intersection of two convex sets).\footnote{This set is non-empty because $\overrightarrow{u}(t)=\overrightarrow{0}$ for all $t$ is a member.} By the Weierstrass theorem \cite[page 7]{liberzon2012calculus}, the continuous function $J(\overrightarrow{x}(T))$ will have a global maximum on this compact set. Thus, by the definition of the reachable set, there exists a control that steers the state to this global maximum, and all such controls will thus be optimal.}
\end{proof}

\section{Proof of Theorem \ref{thm:crossing}}\label{subsec:proof1}

We define {\em co-state} variables corresponding to each of the state variables in \eqref{prob:reformulated}, as summarized in Table \ref{transtable}.
\begin{table}[htb]
\begin{center}
\begin{tabular}{ | c | c | c | c|}
  \hline
  state variable & $\overrightarrow{x}$ &  $\gamma$ &  {$q$} \\
  \hline			
  co-state variable &  $\overrightarrow{\bLam}$ & $\beta$  & {$z$} \\
  \hline  
\end{tabular}
\end{center}
\caption{Correspondence between the state variables and their respective co-state variables.}\label{transtable}
\end{table}
For this new system of equations, we define the Hamiltonian as:
\begin{align}\label{eq:hamiltonian}
\ham(&\overrightarrow{x}(t), \gamma, \mathstr{q,} \overrightarrow{\bLam}, \beta, \mathstr{z,} \overrightarrow{u}(t) \red{, t}) = - \overrightarrow{\bLam}^T \matr{L} \overrightarrow{x} + \overrightarrow{\bLam}^T\matr{B}\overrightarrow{u}+\overrightarrow{\bLam}^T \overrightarrow{e}({q}) - {\beta} \sum_{k=1}^{m}c_k(u_k) {+ z} ,
\end{align}
 where for the continuous {\em co-state} functions we have (at points of continuity of the controls):
\begin{align}
\dot{\overrightarrow{\bLam}}(t) &= -\dfrac{d\ham}{d\overrightarrow{x}} = \matr{L}^T \overrightarrow{\bLam}(t), ~~
\dot{{\beta}}(t) = -\dfrac{d\ham}{d{\gamma}} = 0,\label{eq:costates1}\\
{\dot{z}(t)}&={-\dfrac{d\ham}{d{q}}=-\inp*{\dfrac{d\overrightarrow{e}(q)}{dq}}{\overrightarrow{\bLam}},}\label{eq:costates2}
\end{align}
with terminal state constraints:
\begin{align}
\hspace*{-0.2in}\overrightarrow{\bLam}(T) &= \dfrac{\lambda_0\partial J(\overrightarrow{x}{(T)})}{\partial \overrightarrow{x}(T)}\geq 0, ~ {z(T)= \dfrac{\lambda_0\partial J(\overrightarrow{x}{(T)})}{\partial z(T)}=0,}\label{eq:terminal}\\
\beta(T)&\geq 0,~ \qquad\qquad\qquad \beta(T)[\gamma(T)+r]=0,\label{eq:terminal2}
\end{align}
with $\lambda_0 \in \{0, 1\}$.

For the case of the sigmoid cost function in \eqref{obj:sigmoid}, this means
\begin{align}\label{eq:final_sigmoid}
\lambda_i(T) = \dfrac{\lambda_0\alpha e^{-\alpha (x_i(T) - \theta_i)}}{\big(1+e^{-\alpha_i (x_i(T) - \theta_i)}\big)^2},
\end{align}
while for the linear cost function \eqref{obj:linear}, we will have:
\begin{align}\label{eq:final_linear}
\lambda_i(T) = \lambda_0 p_i.
\end{align}

Pontryagin's Maximum Principle (PMP) \cite[page 182]{seierstad1986optimal} gives us the following necessary conditions for an optimal control \footnote{\green{These necessary conditions are stated for an \emph{autonomous} system, as the dynamics and thus \eqref{eq:hamiltonian} do not depend explicitly on time.}}:

 If: 
\begin{compactitem}
\item $\overrightarrow{u^*}\in \mathcal{U}$ is the piecewise continuous optimal control,
\item $\overrightarrow{x^*}, \gamma^*{, q^*}$ are state trajectories to \eqref{prob:reformulated} under the optimal control,
\item $\overrightarrow{\bLam^*}, {\beta^*}{, z^*}$ are continuous \emph{co-state} real-valued functions on $[0,T]$ that satisfy the terminal time constraints \eqref{eq:terminal} and \eqref{eq:terminal2}, and
dynamics \eqref{eq:costates1} and \eqref{eq:costates2} at points of continuity of the controls,
\end{compactitem}
 then for any control $\overrightarrow{u}\in \mathcal{U}$ that respects \eqref{eq:influencecapital}, 
the following holds: 
\begin{enumerate}
\item For all $0\leq t\leq T$:
\begin{align}\label{eq:hamiltonianmax}
\ham(\overrightarrow{x^*}(t), \gamma^*,  &\mathstr{q^*,} \overrightarrow{\bLam^*}, \beta^*, {z^*,} \overrightarrow{u^*}(t))\geq 
\ham(\overrightarrow{x^*}(t), \gamma^*, \mathstr{q^*,} \overrightarrow{\bLam}^*, \beta^*, {z^*,} \overrightarrow{u}(t)),
\end{align}

\item $(\lambda_0, \overrightarrow{\bLam^*}(t), \beta^*(t){, z^*(t)})\neq \overrightarrow{0}_{n+{3}}$ for all $0\leq t\leq T$.
\end{enumerate}
Since $\dot{\beta^*}(t)=0$ by  \eqref{eq:costates1} for all times at which the controls are continuous, and also due to continuity of the co-states, we obtain:
\begin{align}\label{eqref:betacondition}
 \beta^*(t)=\beta^*(T),~~~~~\forall {t \in [0,T]}.
 \end{align}

For $i\in [m]$, we will use $\varphi_i(t)$ to denote the part of the Hamiltonian \eqref{eq:hamiltonian} that depends explicitly on $u_i$. Therefore for all $i\in [m]$ (using \eqref{eqref:betacondition}):
\begin{align}\label{eq:explicitdependence}
\varphi_{i}(u_i,t):= \inp*{\overrightarrow{\bLam^*}(t)}{\overrightarrow{B}(:,i)} u_i - \beta^*(T) c_i(u_i).
\end{align}

The point-wise Hamiltonian maximizing condition \eqref{eq:hamiltonianmax} of the PMP leads to the fact that $\varphi_{i}(u^*_i,t)\geq \varphi_{i}(u_i,t)$ for all $t$ and all $u_i\in [0, u_i^{\max}(t)]$.
Due to the differentiability of the co-states, $\varphi_i(\cdot)$ is a convex function of the bounded $u_i(t)$, and thus it will be maximized either at one of the upper- or lower-bounds of $u_i(t)$, or it will have the same value for all $u_i\in[0, u_i^{\max}(t)
]$. This is because due to the convexity:
\begin{align*}
\varphi_{i}(u_i,t)&= \varphi_{i}\big((1-\frac{u_i}{u_i^{\max}(t)})*0+ \frac{u_i}{u_i^{\max}(t)}*u_i^{\max}(t),t\big) \leq (1-\frac{u_i}{u_i^{\max}(t)})*\varphi_{i}(0,t)+(\frac{u_i}{u_i^{\max}(t)})* \varphi_{i}(u_i^{\max}(t),t))\\
& \leq (1-\frac{u_i}{u_i^{\max}(t)} + \frac{u_i}{u_i^{\max}(t)})*\max\{\varphi_{i}(0,t),\varphi_{i}(u_i^{\max}(t),t))\} = \max\{\varphi_{i}(0,t),\varphi_{i}(u_i^{\max}(t),t))\}.
\end{align*}

We will now present different arguments for how  
the Hamiltonian maximizing condition is refined for strictly concave and linear $c_k(\cdot)$ functions. The arguments in each case only depend on the nature of $c_k(\cdot)$, and thus we also consider vector function $\overrightarrow{c}(\cdot)$ that simultaneously has elements that are linear and strictly concave:

\begin{enumerate}
\item For the case of strictly concave $c_i(\cdot)$,  Theorem \ref{thm:crossing} only considers cases where $u_i^{\max}(t)=u_i^{\max}$, i.e., is a constant. At each time $t$, $\varphi_{i}(u_i(t),t)=\inp*{\overrightarrow{\bLam^*}(t)}{\overrightarrow{B}(:,i)} u_i(t) - \beta^*(T) c_i(u_i(t))$ is convex in $u_i(t)$. Therefore, it is maximized either at its upper or lower-bound (i.e., $u^*_i(t)\in \{0, u_i^{\max}\}$). Thus, it suffices to compare $\varphi_{i}(0,t)=0$ and $\varphi_{i}(u^{max}_i,t)=\inp*{\overrightarrow{\bLam^*}(t)}{\overrightarrow{B}(:,i)} u^{max}_i - \beta c_i(u^{max}_i)$. Thus, the Hamiltonian maximizing condition \eqref{eq:hamiltonianmax} becomes\footnote{The question mark denoting the fact that PMP does not uniquely determine the optimal $u_i^*$ at times $t$ when $\varphi_i(t, u_i)$ does not change with $u_i$.}:
\begin{align}\label{eq:cases-concave}
\hspace*{-0.2in}u^*_i(t)=
\begin{cases}
u^{max}_i,  &\text{if} ~\inp*{\overrightarrow{\bLam^*}(t)}{\overrightarrow{B}(:,i)}> \beta^*(T)\frac{c_i(u_i^{max})}{u_i^{max}},\\
0,   &\text{if}~ \inp*{\overrightarrow{\bLam^*}(t)}{\overrightarrow{B}(:,i)} < \beta^*(T)\frac{c_i(u_i^{max})}{u_i^{max}},\\
\text{?},  &\text{if}~ \inp*{\overrightarrow{\bLam^*}(t)}{\overrightarrow{B}(:,i)} = \beta^*(T)\frac{c_i(u_i^{max})}{u_i^{max}}.
\end{cases}
\end{align}

\item For the case of linear $c_i(u_i)= v_i u_i$, at each time $t$, $\varphi_{i}(u_i(t), t)=\big(\inp*{\overrightarrow{\bLam^*}(t)}{\overrightarrow{B}(:,i)} - \beta^*(T) v_i\big) u_i $. Thus, the Hamiltonian maximizing condition \eqref{eq:hamiltonianmax} becomes:
\begin{align}\label{eq:cases-linear}
\hspace*{-0.3in}u^*_i(t)=
\begin{cases}
u^{max}_i(t), \quad &\text{if} ~ \inp*{\overrightarrow{\bLam^*}(t)}{\overrightarrow{B}(:,i)} > \beta^*(T) v_i,\\
0, \quad &\text{if} ~  \inp*{\overrightarrow{\bLam^*}(t)}{\overrightarrow{B}(:,i)} < \beta^*(T) v_i,\\
\text{?},\quad &\text{if} ~ \inp*{\overrightarrow{\bLam^*}(t)}{\overrightarrow{B}(:.i)} = \beta^*(T) v_i.
\end{cases}
\end{align}

\end{enumerate}

 As can be seen, the right-hand side of the condition terms in both cases (i.e., $\beta^*(T)c_i(u_i^{\max})$ in \eqref{eq:cases-concave} and $\beta^*(T)v_i$ in \eqref{eq:cases-linear}) is a positive constant. Thus, the structure of the optimal control $u_i^*$ can be understood by examining the fluctuations of $\inp*{\overrightarrow{\bLam^*}(t)}{\overrightarrow{B}(:,i)}$ around two constant values ($\beta^*(T)\frac{c_i(u_i^{max})}{u_i^{max}}$ and $\beta^*(T) v_i$).\footnote{Notice that according to the system model, $c_i(\cdot)$'s can potentially be a mix of linear and strictly concave, and the above statement for any concave or linear $c_i(\cdot)$ is independent of the nature of all other $c_j(\cdot)$'s (even if they are not concave at all). \red{For more details, see Remark \ref{rem:disciplined}}.}

We can trivially show, using the second PMP necessary optimality condition, that the optimal control is \emph{normal}\cite[page 82]{liberzon2012calculus} (i.e., $\lambda_0\neq0$):

\begin{Lemma}\label{lem:normal}
There is no case for which $\lambda_0=0$.
\end{Lemma}

\begin{proof}
If $\lambda_0=0$, from \eqref{eq:terminal}, $\overrightarrow{\bLam^*}(T)=0$. As $\overrightarrow{\bLam^*}(t)=\overrightarrow{0}_n$ for all $t$ is a solution of the differential equation $\dot{\overrightarrow{\bLam^*}}(t)=\matr{L}^T \overrightarrow{\bLam^*}(t)$ with this condition, due to the uniqueness of solutions to differential equations  \cite[Theorem A.8]{seierstad1986optimal} it is also the unique solution. 

Now, {as} $\mathstr{z^*(T)=0}$ { \eqref{eq:terminal},} the only way for \[(\lambda_0, \overrightarrow{\bLam^*}(T), {\beta^*}(T){, z^*(T)})\neq \overrightarrow{0}_{n+{3}}\] (which is the second necessary condition of the PMP) to hold is for ${\beta^*}(t)={\beta^*}(T) \neq { 0}$, which together with \eqref{eq:terminal2} leads to $\beta^*(T)>0$. In this case, for all $i$, and for all $t$, 
\begin{align*}
\varphi_i(u_i, t)&=\inp*{\overrightarrow{\bLam^*}(t)}{\overrightarrow{B}(:,i)} u_i - \beta^*(T) c_i(u_i) = - \beta^*(T) c_i(u_i)\end{align*}
and therefore due to the Hamiltonian maximizing condition \eqref{eq:hamiltonianmax}, $u^*_i(t)=0$ for all such $i$ and all $t$. This leads to $\dot{\gamma^*}(t)=0$ for all $t$, and thus $\gamma^*(T)=0$. Therefore $\beta^*(T)[\gamma^*(T)+r]=\beta^*(T)r >0$, which is a contradiction with \eqref{eq:terminal2}. Therefore $\lambda_0=0$ leads to $(\lambda_0, \overrightarrow{\bLam^*}(T), {\beta^*}(T){, z^*(T)})= \overrightarrow{0}_{n+{3}}$, a contradiction with PMP. 
\end{proof}

Thus, henceforth we replace $\lambda_0$ with 1, and have from \eqref{eq:terminal} and Assumption \ref{ass:positive} that there exists a $j\in [m]$ such that:
\begin{align}\label{eq:lambdapos}
\lambda^*_j(T)>0.
\end{align}

\textbf{Zeros of $\inp*{\overrightarrow{\bLam^*}(t)}{\overrightarrow{B}(:,i)}=\text{constant}$:}

In this subsection, we look at the dynamics of the function $\inp*{\overrightarrow{\bLam^*}(t)}{\overrightarrow{B}(:,i)}$, especially with regards to how many times it can fluctuate around a fixed value in the time-horizon. Any time that this function crosses the fixed value, due to \eqref{eq:cases-concave} and \eqref{eq:cases-linear}, it leads to a switch in the optimal control from one bound to another, and if it is equal to the fixed value over any interval, then the PMP cannot uniquely determine the optimal control (i.e., the optimal control is \emph{singular} \cite[page 113]{liberzon2012calculus}).

Henceforth, let $s\geq 0$ denote the positive constant that $\inp*{\overrightarrow{\bLam^*}(t)}{\overrightarrow{B}(:,i)}$ is set equal to.
\emph{We claim that the expression for $\inp*{\overrightarrow{\bLam^*}(t)}{\overrightarrow{B}(:,i)}= s$, as a function of time $t$, at most has $n-1$ roots.} Therefore due to the continuity of each of the elements of $\overrightarrow{\bLam^*}$, we have that $u^*_i(t)$ is either equal to $u_i^{max}(t)$, or $0$ for all $t$ except maybe at $n-1$ points where it can switch between those two values, proving both parts \eqref{thm:crossing} and \eqref{thm:parta} of the theorem. We now prove this claim:

 We know that $\dot{\overrightarrow{\bLam^*}}=\matr{L}^T \overrightarrow{\bLam^*}$. 
So we have:
\begin{align*}
\dot{\overrightarrow{\bLam^*}}=\matr{Q} \matr{\Xi}^T \matr{Q}^T \overrightarrow{\bLam^*}=\matr{Q} \matr{\Xi} \matr{Q}^T \overrightarrow{\bLam^*},
\end{align*}
as $\matr{\Xi}$ is a diagonal matrix.

We define $\overrightarrow{y}(t):=\matr{Q}^T \overrightarrow{\bLam^*}(t)$. As the columns of $\matr{Q}$ are orthogonal, therefore
\begin{align*}
\dot{\overrightarrow{y}}&= \matr{Q}^T \dot{\overrightarrow{\bLam^*}}= (\matr{Q}^T \matr{Q})\matr{\Xi} \matr{Q}^T \overrightarrow{\bLam^*}
= \matr{\Xi} \matr{Q}^T \overrightarrow{\bLam^*} = \matr{\Xi} \overrightarrow{y},
\end{align*}
with $\overrightarrow{y}(T)= \matr{Q}^T \overrightarrow{\bLam^*}(T)$. Therefore, due to the uniqueness of solutions to ODEs \cite[Theorem A.8]{seierstad1986optimal}, for all \green{$i\in [n]$}:
\begin{align*}
y_i(t)= y_i(T) e^{\xi_{i}(t-T)}= \inp*{\overrightarrow{Q}(:,i)}{\overrightarrow{\bLam^*}(T)}e^{-\xi_{i}(T-t)}.
\end{align*}
Again using the fact that $Q$ is unitary, we have:
\begin{align*}
\overrightarrow{\bLam^*}(t)&= (\matr{Q} \matr{Q}^T) \overrightarrow{\bLam^*}(t) = \matr{Q}\overrightarrow{y}= \matr{Q}\bigg[\inp*{\overrightarrow{Q}(:,i)}{\overrightarrow{\bLam^*}(T)} e^{-\xi_{i}(T-t)}\bigg]_{i} = \sum_{i=1}^n \bigg[\inp*{\overrightarrow{Q}(:,i)}{\overrightarrow{\bLam^*}(T)}e^{-\xi_{i}(T-t)}\bigg] \overrightarrow{Q}(:,i).
\end{align*}

So, we will have:
\begin{align}\label{eq:equation}
\inp*{\overrightarrow{\bLam^*}(t)}{\overrightarrow{B}(:,i)}&= \sum_{j=1}^n \bigg[\inp*{\overrightarrow{Q}(:,j)}{\overrightarrow{\bLam^*}(T)}e^{-\xi_{j}(T-t)}\bigg] \bigg[\inp*{\overrightarrow{Q}(:,j)}{\overrightarrow{B}(:,i)}\bigg]=
 \sum_{j=1}^n \bigg[\inp*{\overrightarrow{Q}(:,j)}{\overrightarrow{\bLam^*}(T)}\inp*{\overrightarrow{Q}(:,j)}{\overrightarrow{B}(:,i)}\bigg]e^{-\xi_{j}(T-t)}\notag\\&=s.
\end{align}

As $\{\overrightarrow{Q}(:,i)\}_i$ is an eigenvector decomposition of $\matr{L}$ and therefore $\xi_1=0$ and $\overrightarrow{Q}(:,1)=\dfrac{1}{\sqrt{n}}\overrightarrow{1}_n$, \eqref{eq:equation} becomes:
\begin{align}\label{eq:equation2}
\sum_{j=2}^n \bigg[&\inp*{\overrightarrow{Q}(:,j)}{\overrightarrow{\bLam^*}(T)}\inp*{\overrightarrow{Q}(:,j)}{\overrightarrow{B}(:,i)}\bigg]e^{-\xi_{j}(T-t)}+ \big(\frac{1}{n} \sum_{k=1}^n \lambda^*_k(T)\sum_{h=1}^n b_{hi}-s\big)=0.
\end{align}

We now determine the number of roots of \eqref{eq:equation2}. We present separate arguments depending on whether any of the coefficients of the exponentials in \eqref{eq:equation2} are non-zero or not.

\textbf{I)} In the case where for some $j>1$, \[\big[\inp*{\overrightarrow{Q}(:,j)}{\overrightarrow{\bLam^*}(T)}\inp*{\overrightarrow{Q}(:,j)}{\overrightarrow{B}(:,i)}\big]\neq 0,\] the following lemma provides the last step in the proof of our claim.

\begin{Lemma}\label{lem:n-1}
For $K>1$, let $\eta_1>\eta_2>\cdots>\eta_K>0$. Also, for $i= 1,\cdots, K$, let $d_{i}$ be a real, non-zero number. Then, the function $f : \mathbb{R} \to \mathbb{R}$, $f(t) = \sum_{i=1}^{K} d_{i} \eta_i^{t}$ has at most $K-1$ zeros.
\end{Lemma}

Setting $\eta_i=e^{\xi_i}$ for all $i$, $d_1=\frac{1}{n} \sum_{k=1}^n \lambda^*_k(T)\sum_{h=1}^n b_{hi}-s$ and $d_j=[\inp*{\overrightarrow{Q}(:,j)}{\overrightarrow{\bLam^*}(T)}\inp*{\overrightarrow{Q}(:,j)}{\overrightarrow{B}(:,i)}] e^{-\xi_{j}T}$ for $j>1$ in this lemma proves that \eqref{eq:equation2} has at most $n-1$ roots. The proof of this lemma is provided in \S Appendix  \ref{sec:lemma2}. Therefore, the number of roots of \eqref{eq:equation2} in this case is bounded above by the number of non-zero elements in $\big\{\inp*{\overrightarrow{Q}(:,j)}{\overrightarrow{\bLam^*}(T)}\inp*{\overrightarrow{Q}(:,j)}{\overrightarrow{B}(:,i)}\big\}_{j=2}^n$, which is itself bounded above by the number of non-zeros in $\big\{\inp*{\overrightarrow{Q}(:,j)}{\overrightarrow{B}(:,i)}\big\}_{j=2}^n$. This is significant because while $\overrightarrow{\bLam^*}(T)$ is not known without the state dynamics (it is a function of the unknown $\overrightarrow{x^*}(T)$, and thus implicitly depends on the trajectory), this bound applies \textbf{to all possible state trajectories (i.e., for all $\overrightarrow{x}(0)$)}. Thus the number of switches in this case is trivially bounded by one less than the number of non-zero elements in $\{\inp*{\overrightarrow{Q}(:,j)}{\overrightarrow{B}(:,i)}\}_{j=1}^n$. 



\textbf{II)} This leaves the case where for  all $2 \leq j\leq n$, $\big[\inp*{\overrightarrow{Q}(:,j)}{\overrightarrow{\bLam^*}(T)}\inp*{\overrightarrow{Q}(:,j)}{\overrightarrow{B}(:,i)}\big]=0$. This means that $\inp*{\overrightarrow{\bLam^*}(t)}{\overrightarrow{B}(:,i)} =\frac{1}{n} (\sum_{k=1}^n \lambda^*_k(T))(\sum_{h=1}^n b_{hi})$, a constant, for all $t$. Now, in order to find the number of roots of \eqref{eq:equation2}, we consider the following two cases:

\textbf{II-1)} If $\frac{1}{n}(\sum_{j=1}^n \lambda_j(T))(\sum_{h=1}^n b_{hi})\neq s$, then \eqref{eq:equation2} has no root and the claim holds.

\textbf{II-2)} Else we have $\frac{1}{n}(\sum_{j=1}^n \lambda_j(T))(\sum_{h=1}^n b_{hi})= s$, meaning that \eqref{eq:equation2} holds for all $t$. We show how this will not be the case for any disciplined channel.

Since $\{\overrightarrow{Q}(:,j)\}_{j=1}^n$ forms an orthonormal basis for $\mathbb{R}^n$ (due to the orthogonality of $\matr{Q}$), this means that we can look at the inner product in this basis: 
\begin{align}
\sum_{j=1}^n \lambda^*_j(T) b_{ji}&= \inp*{\overrightarrow{\bLam^*}(T)}{\overrightarrow{B}(:,i)} = \sum_{j=1}^n \big[\inp*{\overrightarrow{Q}(:,j)}{\overrightarrow{\bLam^*}(T)}\inp*{\overrightarrow{Q}(:,j)}{\overrightarrow{B}(:,i)}\big]
=\big[\inp*{\overrightarrow{Q}(:,1)}{\overrightarrow{\bLam^*}(T)}\inp*{\overrightarrow{Q}(:,1)}{\overrightarrow{B}(:,i)}\big] \notag\\&= \frac{1}{n}(\sum_{j=1}^n \lambda^*_j(T))(\sum_{k=1}^n b_{ki}) ~~~
\Rightarrow ~~~ \sum_{j=1}^n \lambda^*_j(T) \big(b_{ji}- \frac{1}{n}\sum_{k=1}^n b_{ki}\big)=0\label{eq:equation3},
\end{align}
which defines a hyper-plane in the space of $\overrightarrow{\bLam^*}(T)$. But from \eqref{eq:equation2}, we have that:
\begin{align}\label{eq:equation4}
\big(\sum_{j=1}^n \lambda^*_j(T)\big)\big(\frac{1}{n} \sum_{k=1}^n b_{ki}\big) = s,
\end{align}
another hyper-plane. Equations \eqref{eq:equation3} and \eqref{eq:equation4} define the $\overrightarrow{\bLam^*}(T)$-space over which $\inp*{\overrightarrow{\bLam^*}(t)}{\overrightarrow{B}(:,i)^T}=s$ for all $t\in[0,T]$, and thus the Hamiltonian maximizing necessary condition \eqref{eq:hamiltonianmax} does not restrict $u_i^*$ over this interval. Such trajectories are known as \emph{singular arcs} \cite[page 113]{liberzon2012calculus}. 

We present different arguments depending on $c_i(\cdot)$:

\textbf{II-2-a)} If $c_i(\cdot)$ is strictly concave, then we utilize the generalized Legendre-Clebsch necessary condition of optimality on singular arcs \cite{longuski2014optimal}: we must have $\dfrac{d^2\ham}{d\overrightarrow{u}^2}\preceq 0$. However, $\dfrac{d^2\ham}{du_i^2}= -\beta^*(T) c''(u_i^*)\geq 0$, due to the strict concavity of $c_i(\cdot)$. Thus, we must have $\beta^*(T)=0$. So \eqref{eq:equation4} becomes:
$\big(\sum_{j=1}^n \lambda^*_j(T)\big)\big(\frac{1}{n} \sum_{k=1}^n b_{ki}\big) = 0$.
This means that either $\sum_{j=1}^n \lambda^*_j(T)=0$ (a contradiction with \eqref{eq:lambdapos}) or $\sum_{k=1}^n b_{ik}=0$ (ruled out for disciplined channels, the ones considered by this theorem, by Definition \ref{def:disciplined}). Thus we have a contradiction, meaning that this case will never arise.

\textbf{II-2-b)} If $c_i(\cdot)$ is linear, we have: $\varphi_{i}(u_i(t), t)=(\inp*{\overrightarrow{\bLam^*}(t)}{\overrightarrow{B}(:,i)}  - \beta^*(T) v_i) u_i $. For the singular arc case, we must have $\Phi(t):=\inp*{\overrightarrow{\bLam^*}(t)}{\overrightarrow{B}(:,i)}  - \beta^*(T) v_i =0$ over all $t$. Thus, as this function is maximally flat, all its derivatives with respect to time must also be zero (infinite-order singularity). We must have:
\begin{align*}
\dfrac{d\Phi(t)}{dt}&= \inp*{\dot{\overrightarrow{\bLam^*}}(t)}{\overrightarrow{B}(:,i)}= \inp*{\matr{L}^T\overrightarrow{\bLam^*}(t)}{\overrightarrow{B}(:,i)} = \overrightarrow{\bLam^*}^T(t)\matr{L}\overrightarrow{B}(:,i)=0.
\end{align*}
Using an induction, we can see that for $l\geq 1$:
\begin{align*}
\dfrac{d^l\Phi(t)}{dt^l}=  \inp*{(\matr{L}^T)^l\overrightarrow{\bLam^*}(t)}{\overrightarrow{B}(:,i)}=\overrightarrow{\bLam^*}^T(t)\matr{L}^l\overrightarrow{B}(:,i)=0.
\end{align*}
So we must have:
\begin{align*}
\inp*{\overrightarrow{\bLam^*}(t)}{\bigg[\matr{L}\overrightarrow{B}(:,i) | \ldots |\matr{L}^n\overrightarrow{B}(:,i) \bigg]}= 0.
\end{align*}
But {we have $\lambda^*_j>0$ for some $j$} (due to \eqref{eq:lambdapos}), so $\overrightarrow{\bLam^*(T)}\neq \overrightarrow{0}_n$. This, however, means that we must have $\text{rank}(\bigg[\matr{L}\overrightarrow{B}(:,i) | \ldots |\matr{L}^{n}\overrightarrow{B}(:,i) \bigg])<n$. However, since the system $(\matr{L}, \matr{L}\overrightarrow{B}(:,i))$ is controllable, this matrix must have row rank $n$ (due to the necessary and sufficient condition of controllability of linear systems \cite[Theorem 6.1]{chen1995linear}). This is a contradiction, meaning that the singular case will not arise in this case either.

This concludes the proof of parts \eqref{thm:crossing} and \eqref{thm:parta} of the theorem. We now proceed to part (\ref{thm:partb}) of the theorem.

When $J(\overrightarrow{x}(T))=\inp*{\overrightarrow{p}}{\overrightarrow{x}(T)}$, from \eqref{eq:final_linear}, we have $\overrightarrow{\bLam^*(T)} = \overrightarrow{p}$. Thus, we can refine the result stated in part (\ref{thm:parta}) using this additional information into the result in (\ref{thm:partb}). We will, however, require differing, stronger arguments for argument I after \eqref{eq:equation2} to obtain the tighter bound on the number of switches (as all the cases in II are either trivial or are shown not to arise). Thus, if we prove that:
\begin{align}\label{eq:equation7}
\sum_{j=2}^n \bigg[\inp*{\overrightarrow{Q}(:,j)}{\overrightarrow{p}}&\inp*{\overrightarrow{Q}(:,j)}{\overrightarrow{B}(:,i)}\bigg]e^{-\xi_{j}(T-t)}+ \big(\frac{1}{n} \sum_{k=1}^n p_k \sum_{h=1}^n b_{hi}-s\big)=0,
\end{align} 
has at most $\{\sum_{k=1}^j s_k \}_{j=1}^n$ zeros, where \[s_j : = \inp*{\overrightarrow{Q}(:,j)}{\overrightarrow{p}} \inp*{\overrightarrow{Q}(:,j)}{\overrightarrow{B}(:,i)},\] then we are done. This can be seen to be true due to the following generalization of Descartes' rule of signs (due to \cite[Theorem 4.7]{jameson2006counting}), as applied to \eqref{eq:equation7}, with $d_i$ defined as in the application of Lemma \ref{lem:n-1}.
\begin{Lemma}\label{lemma:descartes2}
The number of positive zeros of the {exponential} polynomial function $f : \mathbb{R} \to \mathbb{R}$, $f(t) = \sum_{i=1}^{n} d_{i} (e^{-t})^{\xi_i}$ is upper-bounded by the number of variations in sign in the sequence $\{s_{i}\}_i$, where $s_i=\sum_{j=1}^i d_j$. 
\end{Lemma}

\section{Proof of Lemma \ref{lem:n-1}}\label{sec:lemma2}

\begin{proof}
By induction on $K$.

\textbf{I)} $K=1$:  $f(t) =  d_{1} \eta_1^{t}$, which does not have a root as $d_1\neq 0$. Therefore the base case holds.

\textbf{II)} $K=k \to K=k+1$: 
\begin{align*}
f(t) &= \sum_{i=1}^{k+1} d_{i} \eta_i^{t} = \eta_{k+1}^t (d_{k+1} + \sum_{i=1}^{k} d_{i} (\dfrac{\eta_i}{\eta_{k+1}})^{t})= \eta_{k+1}^t g(t),
\end{align*}
 with every zero of $f(t)$ being a zero of $g(t)$. Now:
\begin{align*}
 &g'(t) =  \sum_{i=1}^{k} \big[ d_{i} \ln(\dfrac{\eta_i}{\eta_{k+1}}) \big] \big(\dfrac{\eta_i}{\eta_{k+1}}\big)^{t}.
\end{align*}        
 Notice that all these coefficients are non-zero due to the statement of the theorem. Thus, due to the induction hypothesis, $g'(t)$ has at most $k-1$ zeros. 
 
We complete the proof by appealing to Rolle's theorem \cite[p. 184, Theorem 4.4]{apostol2007calculus}:
\begin{Theorem}[Rolle's theorem]\label{thm:rolle}
If $h(\cdot)$ is a continuous-everywhere function on $[a,b]$ and has a derivative at each point in $(a, b)$ and $h(a)=h(b)$, then there exists $c \in (a,b)$ such that $h'(c)=0$.
\end{Theorem}
If $g(t)$ has strictly more than $k$ zeros, then by the theorem, $g'(t)$ will have strictly more than $k-1$ zeros, a contradiction. Thus,  $g(t)$, and thus $f(t)$, can have at most $k$ zeros, completing the proof of the lemma.
\end{proof}

\section{Proof of Lemma \ref{lemma:descartes2} condensed from \cite{jameson2006counting}}\label{sec:lemma3}
\begin{proof}
We first state and prove two lemmas:
\begin{Lemma}\label{lem:multiplicity}
Suppose $g(\cdot)$ has a zero of order $k$ at $t=t_0$. Let $h(\cdot)$ be another function such that $h(t_0)\neq 0$. Then $g(\cdot)h(\cdot)$ has a zero of order $k$ at $t=t_0$.
\end{Lemma}
\begin{proof}
For all $n$, we have:
\[\frac{d^n}{dt^n}(g(t)h(t))=\sum_{i=0}^{n-1} {n \choose i} \frac{d^ig(t)}{dt^i} \frac{d^{n-i}h(t)}{dt^{n-i}} + \frac{d^ng(t)}{dt^n}h(t).\]
For $n=1,\ldots, k$ both terms on the right-hand side are zero at $t=t_0$ due to the order of the $g(\cdot)$ zero. However, for $n=k+1$, while all the terms in the sum again be zero, as $h(t_0)\neq 0$ and $\dfrac{d^{k+1}g(t_0)}{dt^{k+1}}\neq 0$, then 
\[\dfrac{d^{k+1}}{dt^n{k+1}}(g(t)h(t))\neq 0,\]
completing the proof of the lemma.
\end{proof}
For every real function $g(\cdot)$, we define $Z_+(g)$ to be the number of zeros of $g(\cdot)$ over $t\in (0,\infty)$ (counted with their potential multiplicities).
\begin{Lemma}\label{lem:zeros}
Suppose $\iota(x)$ is bounded, continuous, and non-zero on each interval $(x_{i-1}, x_i)$ such that $-\infty < x_1 <\ldots < b=x_n$.
We will count $x_i$ as a sign change if the sign of $\iota(x)$ is different in $(x_{i-1}, x_i)$ and $(x_{i}, x_{i+1})$. Then, if $g(t)= \int_{-\infty}^b \iota(x) e^{tx}dx$ and $m$ is the number of sign changes of $\iota(x)$ in the interval $(-\infty, b)$, $Z_+(g)\leq m$.
\end{Lemma}
\begin{proof}
We prove $Z_+(g)\leq m$ by induction on $m$:

$\mathbf{m=0}$: If $\iota(x)>0$ for all $x<-\xi_1$, it will also be positive in all intervals $(x_{i}, x_{i+1})$. This means that $g(t)>0$ for all such $t$, and therefore $Z_+(g)=0$. The same reasoning applies to $\iota(x)<0$ over $x<-\xi_1$.

We now assume the result for $m=k$ and \textbf{prove the $m=k+1$ case}.
If the number of sign changes of $\iota(x)$ is $k+1$, without loss of generality we assume the last zero-crossing happens at $x_k$. So $\iota(x) > 0$ over  $x \in (x_{k-1}, x_{k})$ and $\iota(x)< 0$ over  $x \in (x_{k}, x_{k+1})$. Define $h(t):= e^{-x_k t} g(t)= \int_{-\infty}^{b} \iota(x)e^{t(x-x_k)} dx$. From Lemma \ref{lem:multiplicity}, $Z_+(h)=Z_+(g)$.

We now look at $k(t):=\dfrac{dh(t)}{dt}$ (using the Leibniz integral rule to change the order of differentiation and integration):
\begin{align}
k(t)= \int_{-\infty}^{b} (x-x_k)\iota(x)e^{t(x-x_k)} dx.
\end{align}
The function $I(x)= (x-x_k)\iota(x)$ fulfills the conditions of the lemma for the same set of $\{x_i\}$ points, with the exception that it will not change signs at $x_k$ (as both terms in $I(x)$ will change signs at that point). Thus, the total number of sign changes of $I(x)$ over $(-\infty, b)$ will be $k$, and by the induction step, $Z_+(k)\leq k$. Using the same Rolle's theorem (Theorem \ref{thm:rolle}) argument in the proof of Lemma \ref{lem:n-1}, we can see that thus $Z_+(h)\leq k+1$, completing the proof of this lemma.
\end{proof}
We now proceed with the proof of Lemma \ref{lemma:descartes2}.
 Define $S(\{s_{i}\}_i)$ to be the number of sign variations in $\{s_{i}\}_i$.
Note:
\begin{align*}
f(t) &= \sum_{i=1}^{n} d_{i} (e^{-t})^{\xi_i} =  \sum_{i=1}^{n-1}  s_i \big((e^{-t})^{\xi_{i}} - (e^{-t})^{\xi_{i+1}}\big) + s_n (e^{-t})^{\xi_{n}}  
\end{align*}
Furthermore, because $\frac{de^{tx}}{dt}= x e^{tx}$: 
\begin{align}\label{eq:ft}
f(t) &= \sum_{i=1}^{n-1}  s_i \big((e^{-t})^{\xi_{i}} - (e^{-t})^{\xi_{i+1}}\big) + s_n (e^{-t})^{\xi_{n}} = \sum_{i=1}^{n-1} s_i \int_{-\xi_{i+1}}^{-\xi^{i}} t e^{tx} dx + s_n \int_{-\infty}^{-\xi^{n}} t e^{tx} dx \notag\\&= t \big(\sum_{i=1}^{n-1}  \int_{-\xi_{i+1}}^{-\xi^{i}} s_i e^{tx} dx +  \int_{-\infty}^{-\xi^{n}} s_n e^{tx} dx\big)
\end{align}
Define $g(t):=\sum_{i=1}^{n-1} \int_{-\xi_{i+1}}^{-\xi^{i}} s_i  e^{tx} dx + \int_{-\infty}^{-\xi^{n}} s_n e^{tx} dx$. Then by Lemma \ref{lem:multiplicity}, $Z_+(f)= Z_+(g)$. Define $\iota(x) := s_i$ for $-\xi_{i+1}< x < -\xi_{i}$ and $\iota(x):=s_n$ for $x < -\xi_{n}$.\footnote{The value of $\iota(x)$ can be arbitrarily assigned to a bounded value at the overlap points.} Then, from \eqref{eq:ft}:
\begin{align}\label{eq:ft1}
g(t) &= \int_{-\infty}^{-\xi^{1}} \iota(x)e^{tx} dx.
\end{align}
But $\iota(x)$ fulfills the conditions of Lemma \ref{lem:zeros} with $S(\{s_i\})_i$ number of points where its sign changes. Thus, by that lemma, $Z_+(g)\leq S(\{s_i\})_i$, and therefore $Z_+(f)\leq S(\{s_i\})_i$, completing the proof of this lemma.

\end{proof}

\green{\section{Proof of Lemma \ref{lem:N-1}}\label{sec:lemma_gen}

\begin{proof}
By induction on M.

\textbf{I)} $M=1$:  $f(t) = \sum_{j=0}^{w-1} d_{j} t^j\eta^{t} = \eta^t g(t)$, with $f(t)=0$ only if $g(t)=0$. Due to the fundamental theorem of algebra, $g(t)$ has at most $w-1 = K-1$ real zeros (as $d_{w-1}>0$), and therefore the base case holds.

\textbf{II)} $M=k \to M=k+1$: 
\begin{align*}
f(t) &= \sum_{i=1}^{k+1}\sum_{j=0}^{w_{i}-1} d_{i,j} t^j\eta_i^{t} = \eta_{k+1}^t (\sum_{j=0}^{w_{k+1}-1} d_{k+1,j} t^j + \sum_{i=1}^{k}\sum_{j=0}^{w_{i}-1} d_{i,j} t^j (\dfrac{\eta_i}{\eta_{k+1}})^{t})= \eta_{k+1}^t g(t),
\end{align*}
 with every zero of $f(t)$ being a zero of $g(t)$. Now:
\begin{align*}
 g'(t) &= \sum_{j=1}^{w_{k+1}-1} d_{k+1,j} j t^{j-1} + \sum_{i=1}^{k}\sum_{j=0}^{w_{i}-1} d_{i,j} j t^{j-1} (\dfrac{\eta_i}{\eta_{k+1}})^{t} + \sum_{i=1}^{k}\sum_{j=0}^{w_{i}-1} d_{i,j} t^{j} (\dfrac{\eta_i}{\eta_{k+1}})^{t} \ln(\dfrac{\eta_i}{\eta_{k+1}}) \\&= \sum_{j=0}^{w_{k+1}-2}  d_{k+1,j+1}({j+1}) t^{j} + \sum_{i=1}^{k} d_{i,{w_{i}-1}} t^{{w_{i}-1}} (\dfrac{\eta_i}{\eta_{k+1}})^{t} \ln(\dfrac{\eta_i}{\eta_{k+1}})+ \sum_{i=1}^{k}\sum_{j=0}^{w_{i}-2} d_{i,{j+1}} t^{j} (\dfrac{\eta_i}{\eta_{k+1}})^{t} (j+1 + \ln(\dfrac{\eta_i}{\eta_{k+1}})).
\end{align*}        
 Differentiating $w_{k+1}-1$ more times, the first term vanishes, and the other terms only change in their coefficients. Specifically, the coefficient of the $t^{w_{k}-1} (\dfrac{\eta_k}{\eta_{k+1}})^{t}$ term will be $d_{k,{w_{k}-1}} \big(\ln(\dfrac{\eta_i}{\eta_{k+1}})\big)^{w_{k+1}}$, which is non-zero due to the statement of the theorem. Thus, due to the induction hypothesis, $\dfrac{\partial^{w_{k+1}-1}g(t)}{\partial t^{w_{k+1}-1}}$ has at most $\sum_{i=1}^{k} w_i -1$ zeros. 
 
We now complete the proof by repeatedly appealing to Rolle's theorem \cite[p. 184, Theorem 4.4]{apostol2007calculus}:
\begin{Theorem}[Rolle's theorem]\label{thm:rolle1}
If $h(\cdot)$ is a continuous-everywhere function on $[a,b]$ and has a derivative at each point in $(a, b)$ and $h(a)=h(b)$, then there exists $c \in (a,b)$ such that $h'(c)=0$.
\end{Theorem}
If $\dfrac{\partial^{w_{k+1}-2}g(t)}{\partial t^{w_{k+1}-2}}$ has strictly more than $\sum_{i=1}^{k} w_i$ zeros, then by the theorem, $\dfrac{\partial^{w_{k+1}-1}g(t)}{\partial t^{w_{k+1}-1}}$ will have strictly more than $\sum_{i=1}^{k} w_i -1$ zeros, a contradiction. Thus, $\dfrac{\partial^{w_{k+1}-2}g(t)}{\partial t^{w_{k+1}-2}}$ will have at most $\sum_{i=1}^{k} w_i$ zeros. Applying the same reasoning reasoning $w_{k+1}-1$ more times shows that $g(t)$, and thus $f(t)$, can have at most $\sum_{i=1}^{k+1} w_i -1$ zeros, completing the proof of the theorem.
\end{proof}}





\vskip 0pt plus -1fil

\begin{IEEEbiographynophoto}{Soheil Eshghi} (M'08)
received his Ph.D. degree
in Electrical and Systems Engineering 
from the University of Pennsylvania
in 2015. He is currently a postdoctoral associate at Yale Institute for Network Science (YINS) and the Department of Electrical Engineering at Yale University.
His research interests include mathematical models of influence in social systems, as well as applications of optimal control and game theory to social, biological, and electrical networks.
\end{IEEEbiographynophoto}

\begin{IEEEbiographynophoto}{Victor M. Preciado}
received the PhD degree in electrical engineering and computer science from the Massachusetts Institute of Technology. He is currently an associate professor of electrical and systems engineering with the University of Pennsylvania. He is a member of the Networked \& Social Systems Engineering program, the Warren Center for Network \& Data Sciences, and the Applied Math \& Computational Science (AMCS) program. He is a recipient of the 2017 National Science Foundation Faculty Early Career Development (CAREER) Award. His main research interests lie at the intersection of big data and network science; in particular, in using innovative mathematical and computational approaches to capture the essence of complex, high-dimensional dynamical systems. Relevant applications of this line of research can be found in the context of socio-technical networks, brain dynamical networks, healthcare operations, biological systems, and critical technological infrastructure. He is a member of the IEEE.
\end{IEEEbiographynophoto}

\begin{IEEEbiographynophoto}{Saswati Sarkar}
received ME from the Electrical
Communication Engineering Department
at the Indian Institute of Science,
Bangalore in 1996 and PhD from the Electrical
and Computer Engineering Department at the
University of Maryland, College Park, in 2000.
She joined the Electrical and Systems Engineering
Department at the University of
Pennsylvania, Philadelphia as an Assistant
Professor in 2000 where she is currently a
Professor. She received the Motorola gold
medal for the best masters student in the
division of electrical sciences at the Indian Institute of Science and a
National Science Foundation (NSF) Faculty Early Career Development
Award in 2003. She was an associate editor of IEEE Transaction on
Wireless Communications from 2001 to 2006, and is currently an associate
editor of IEEE/ACM Transactions on Networks. Her research interests
are in stochastic control, resource allocation, dynamic games and economics
of networks.
\end{IEEEbiographynophoto}

\begin{IEEEbiographynophoto}{Santosh S. Venkatesh} (M'86) received
the Ph.D. degree in electrical engineering from
the California Institute of Technology, Pasadena,
in 1986. He is a Professor with the
Department of Electrical and Systems Engineering,
University of Pennsylvania, Philadelphia.
His research interests are in probabilistic
models, random graphs, epidemiology, and
network and information theory.
\end{IEEEbiographynophoto}

\begin{IEEEbiographynophoto}{Qing Zhao}
(F'13) received the Ph.D. degree in Electrical Engineering from Cornell University in 2001. She is currently a Professor and Gordon Lankton Sesquicentennial Faculty Fellow in the School of Electrical
and Computer Engineering at Cornell University, Ithaca, NY. Prior to joining Cornell University in 2015,  she was
a Professor at University of California, Davis. Her research interests include
sequential decision theory and stochastic optimization, machine learning, statistical inference, and algorithmic theory with applications in infrastructure and communication systems and social economic networks. She is a Fellow of IEEE. She received the 2010 IEEE Signal Processing Magazine
Best Paper Award and the 2000 Young Author Best Paper Award from
the IEEE Signal Processing Society.
While on the faculty of UC Davis, she held the title of UC Davis Chancellor's Fellow.
\end{IEEEbiographynophoto}

\begin{IEEEbiographynophoto}{Raissa D'Souza} received the Ph.D. degree in statistical physics from the Massachusetts Institute of Technology (MIT), Cambridge, MA, USA, in 1999. She is a Professor of Computer Science and of Mechanical Engineering at the University of California Davis, Davis, CA, USA, as well as an External Professor at the Santa Fe Institute, Santa Fe, NM, USA. She is a Fellow of the American Physical Society, serves on the editorial board of numerous international mathematics and physics journals, has organized key scientific meetings like NetSci 2014, was a member of the World Economic Forum's Global Agenda Council on Complex Systems, and is currently the President of the Network Science Society.
\end{IEEEbiographynophoto}

\begin{IEEEbiographynophoto}{Ananthram Swami} (F'08)
is with the U.S. Army Research
Laboratory (ARL) as the Army's ST (Senior Research
Scientist) for Network Science. He is an ARL Fellow and
Fellow of the IEEE. He has held positions with Unocal
Corporation, the University of Southern California (USC),
CS-3 and Malgudi Systems. He was a Statistical
Consultant to the California Lottery, developed a MATLAB-based toolbox for
non-Gaussian signal processing, and has held visiting faculty positions at INP,
Toulouse, and Imperial College, London. He received the B.Tech. degree from
IIT-Bombay, the M.S. degree from Rice University, and the Ph.D. degree from
USC, all in Electrical Engineering. His research interests are in the broad area of network science with applications in composite tactical networks.
\end{IEEEbiographynophoto}

\end{document}